\documentclass[12pt]{amsart}
\usepackage{amsmath}
\usepackage{enumerate}
\usepackage{fancyvrb}
\usepackage{amssymb}
\usepackage{color}
\usepackage{mathrsfs}

\usepackage{mathabx}
\usepackage{amsthm}
\usepackage[x11names]{xcolor}
\usepackage{wasysym}

\usepackage[backref=page]{hyperref}
\hypersetup{colorlinks=true,linkcolor=blue,citecolor=red}

\usepackage[margin=1.35 in, includeheadfoot=true]{geometry}

\usepackage{tikz}
\usepackage{scalerel}
\usepackage{pict2e}
\usepackage{tkz-euclide}
\usetikzlibrary{calc}
\usetikzlibrary{patterns,arrows.meta}
\usetikzlibrary{shadows}
\usetikzlibrary{external}

\usepackage{pgfplots}
\pgfplotsset{compat=newest}
\usepgfplotslibrary{statistics}
\usepgfplotslibrary{fillbetween}
\usetikzlibrary{cd}
\usetikzlibrary{decorations.markings}

\theoremstyle{plain}
\newtheorem{theorem}{Theorem}[section]
\newtheorem{prop}[theorem]{Proposition}
\newtheorem{thm}[theorem]{Theorem}
\newtheorem*{thm*}{Theorem}
\newtheorem{lemma}[theorem]{Lemma}

\theoremstyle{definition}
\newtheorem{defn}[theorem]{Definition}
\newtheorem{rmk}[theorem]{Remark}
\newtheorem{example}[theorem]{Example}

\def\C{\mathbb{C}}

\def\Z{\mathbb{Z}}

\def\I{\mathcal{I}}
\def\j{\mathbf{j}}

\def\K{\mathcal{K}}
\def\O{\mathcal{O}}

\def\del{\operatorname{del}}
\def\link{\operatorname{link}}

\def\span{\operatorname{Span}}

\def\init{\operatorname{init}}
\def\ord{\operatorname{ord}}

\def\Im{\operatorname{Im}}
\def\id{\operatorname{id}}

\def\ess{\operatorname{Ess}}
\def\LT{\operatorname{LT}}
\def\Dem{\operatorname{Dem}}
\def\bsswap{\mathsf{bsswap}}

\newcommand{\tup}[1]{\langle{#1}\rangle}

\title[A GB for Kazhdan-Lusztig ideals of the flag variety of affine type $A$]{A Gr\"obner basis for  Kazhdan-Lusztig ideals of the  flag variety of affine type $A$}
\author{Bal\'azs Elek}
\author{Daoji Huang}
\thanks{Daoji Huang was partially supported by NSF grant DMS-2202900}

\date{}
\begin{document}
\maketitle
\begin{abstract}
A Kazhdan-Lusztig variety is the intersection of a locally-closed Schubert cell with an opposite Schubert variety in a flag variety. We present a linear parametrization of the Schubert cells in the affine type A flag variety via Bott-Samelson maps, and give explicit equations that generate the Kazhdan-Lusztig ideals in these coordinates. Furthermore, our equations form a Gröbner basis for the Kazhdan-Lusztig ideals. Our result generalizes a result of Woo-Yong that gave a Gröbner basis for Kazhdan-Lusztig ideals in the type $A$ flag variety.
\end{abstract}
\section{Introduction}

In \cite{woo2012grobner}, Woo and Yong gave a Gr\"obner basis for Kazhdan-Lusztig ideals in the type $A$ flag variety, generalizing the Gr\"obner basis theorem \cite[Theorem B]{knutson2005grobner} for Schubert determinantal ideals. In this paper, we present a Gr\"obner basis for Kazhdan-Lusztig ideals in the affine type $A$ flag variety, generalizing the main result of \cite{woo2012grobner}. 

Our original motivation for finding explicit defining equations for Kazhdan-Lusztig varieties in the flag variety of affine type $A$ comes from the theory of Bruhat/Kazhdan--Lusztig atlases for projected Richardson varieties in the partial flag varieties. Projected Richardson varieties \cite{knutson2014projections} are projections of Richardson varieties from $G/B$ to $G/P$, where $G$ is a reductive algebraic group over an algebraically closed field, $B$ a Borel subgroup and $P$ a parabolic subgroup.  The stratification of $G/P$ by projected Richardson varieties has many nice properties; for example, the closed strata are are exactly the compatibly split subvarieties with respect to the standard Frobenius splitting on $G/P$. When $G/P$ is the Grassmannian $Gr(k,n)$, the projected Richardson varieties are also known as positroid varieties. One way to study this stratification is to cover $G/P$ with standard affine charts $U_f$ (i.e., permuted big cells) at torus-fixed points $f$, and it was shown in \cite{galashin2022regularity,KLchart} that there exist Schubert cells $\mathcal{X}^{w(f)}_\circ $ stratified by Kazhdan-Lusztig varieties $\mathcal{X}^{w(f)}_\circ\cap \mathcal{X}_v$ in the affine flag variety $\mathcal{G}/\mathcal{B}$ of the loop group $\mathcal{G}$ associated to $G$, such that $U_f$ and $\mathcal{X}^{w(f)}_\circ $ are stratified-isomorphic. The Bruhat atlas for $Gr(k,n)$ stratified by positroid varieties was first discovered in \cite{Snider}. In \cite{galashin2022regularity}, this atlas served as a crucial technical tool to establish topological properties of the totally nonnegative part of a partial flag variety. In the affine type $A$ setting, our result makes it possible to practically compute with equations defining Kazhdan-Lusztig varieties using computer software, e.g. \texttt{Macaulay2}, which gives computational access for studying local equations of positroid varieties and more generally, projected Richardson varieties on $d$-step flag varieties. 

The geometry of Kazhdan-Lusztig varieties in (Kac--Moody) flag varieties is also closely related to formulas that compute the restriction of torus-equivariant 
cohomology/$K$-theoretic Schubert classes to fixed points. Since Kazhdan-Lusztig varieties degenerate to subword complexes, their Hilbert series can be computed explicitly using combinatorics of reduced words. Under an appropriate grading induced by the torus action, their multidegree/$K$-polynomials can recover the Andersen–Jantzen–Soergel/Billey \cite{andersen1994representations,billey} and Graham/Willems \cite{Gra02, willems2004equivariant} restriction formulas. Our Gr\"obner basis completes the story that explicitly interprets these restriction formulas using Gr\"obner geometry in affine type $A$. For details, see Remark~\ref{rmk:equivariant}.

There is extensive literature on the geometry of Kazhdan-Lusztig varieties in the type $A$ setting. In \cite{LI2012633}, Li and Yong studied Hilbert–Samuel multiplicity for points of Schubert varieties in the complete flag variety by Gr\"obner degeneration of  Kazhdan–Lusztig ideals, and gave an explicit combinatorial interpretation in terms of subword complexes in the (co)vexillary case. This result was generalized by Anderson--Ikeda--Jeon--Kawago to classical types \cite{ANDERSON2023109366}.  The potential relationship between the $h$-polynomials for Kazhdan-Lusztig varieties and the celebrated Kazhdan-Lusztig polynomials was explored in \cite{li2012kazhdan}. Outside of type $A$, the Gr\"obner geometry of Kazhdan--Lusztig varieties in type $C$ has been studied by Escobar--Fink--Rajchgot--Woo \cite{escobar2024grobner}. We hope our work will provide useful tools for similar investigations in the affine type $A$ case. 

Unlike in the finite type $A$ case, the technology that expresses Schubert conditions in the flag variety as determinantal conditions on matrices was not readily available in affine type $A$. Therefore, to achieve our main result, we develop the following tools:
\begin{itemize}
    \item[(A)] An analogue of Fulton's ideals for matrix Schubert varieties \cite{fulton92} in the space of infinite periodic matrices, expressing  finite codimensional (opposite) Schubert conditions in the affine type $A$ flag variety (Theorem~\ref{thm:seteq});
    \item[(B)] A linear parametrization using the Bott-Samelson coordinates for the finite dimensional Schubert cells (Proposition~\ref{prop:linearBS}).
\end{itemize}
We obtain the defining equations for Kazhdan-Lusztig varieties by pulling back equations from (A) to coordinates determined by (B). Following the strategies developed in \cite{woo2012grobner}, we prove that these equations form a  Gr\"obner basis of the Kazhdan-Lusztig ideals by an inductive argument on subword complexes. Our proof relies on a result of Knutson, stated in Theorem~\ref{thm:knutson}, that degenerate Kazhdan-Lusztig ideals expressed in Bott-Samelson coordinates to Stanley-Reisner ideals for subword complexes.


Outside of the finite type settings, the singularities of finite-dimensional Schubert varieties in the affine Grassmannian, in type $A$ as well as all types, have been widely studied, see e.g., \cite{malkin, kuttler2009singularities, billey2010smooth, haines}. For Schubert varieties in general Kac-Moody flag varieties, a criterion for smoothness is given in \cite{KumarShrawan1996TnHr}.
On the other hand, the Gr\"obner geometry in the affine type settings has not been extensively studied by the literature. In the thesis \cite{brunson2014matrix}, Brunson investigated (set-theoretic)  equations for finite-dimensional Schubert conditions in certain matrix coordinates for the affine Grassmannian, extending the work of \cite{kreiman2004ideal}. Their work is related to but different from ours. One distinction is the finite dimensional versus codimensional conditions, which is not essentially different in the finite type case, but significantly different in the affine case. Our work shows that it is possible to work with finite-codimensional Schubert conditions explicitly via restrictions to finite-dimensional cells.
\vskip 1em

\noindent\textbf{Acknowledgements.} The authors would like to thank Shiliang Gao, Allen Knutson, Mark Shimozono, and Alex Yong for helpful conversations. We thank the anonymous referees for their very careful reading and  many helpful suggestions.

\section{Matrix realization of the affine flag variety in type A}
Our main references for this section
are \cite{magyar2002affine}, \cite{kumar2012kac}, and \cite{galashin2022regularity}.
Let $\mathcal{K}:=\C((t^{-1}))$ denote the  field of formal Laurent series, and 
$\mathcal{O}:=\C[[t^{-1}]]$ the ring of formal power series\footnote{Our choice of using $t^{-1}$ instead of $t$ as the generator is the opposite of the standard choice in the literature. This choice was made so that in our matrix model, the affine type $A$ convention is consistent with the literature on matrix Schubert varieties.}.
For $f\in \K$ where $f(t) =\sum_{i\ge N}a_i t^{-i}$ with $a_i\in \C$, and $f\neq 0$, we let
$\ord(f)$ be the smallest integer for which $a_i\neq 0$.
Fix a positive integer $n$;
let $G(\K):=GL_n(\K)$ and $G(\O):=GL_n(\O)$.  

Let $\{e_1,\cdots, e_n\}$ denote the standard $\K$-basis of $V:=\K^n$, and
for $c\in \Z$, define $e_{j+nc}:= t^c e_j$.
A \textbf{lattice} $L \subset V$
is an $\O$-submodule $L= \span_\O(v_1,\ldots , v_n)$ where
$\{v_1,\cdots ,v_n\}$ is a $\K$-basis of $V$. 
Consider the family of \textbf{standard $\O$-lattices}
\[E_j := \span_\O\tup{e_j,e_{j-1},\cdots, e_{j-n+1}} = \mathrm{fSpan}_\C\tup{e_k}_{k\le j},\]
where $\mathrm{fSpan}$ denotes the formal span of possibly infinitely many basis vectors. 
Note that $E_j =\delta^j E_1$, where 
$\delta$ is the shift operator defined as
 $\delta(e_j)= e_{j+1}$, or as a matrix
\[\delta = \begin{bmatrix} 0 & 0 &\cdots &0 & t \\
                           1 & 0 &\cdots &0 & 0 \\
                           0 & 1 &\cdots &0 & 0 \\ 
                           \vdots & \vdots & & \vdots &\vdots \\
                           0 & 0 & \cdots & 1 & 0\end{bmatrix}.\]
Let $G_k:=\{g\in G(\K) :\text{ord}\det(g)=-k\}$, so 
$G_kG_\ell=G_{k+\ell}$. 
Note that
$\delta\in G_1$ and $G_k=\delta^k G_0=G_0\delta^k$.
The decomposition
$G(\K)=\coprod_{k\in\Z} G_j$ can be thought of as decomposing
$G(\K)$ into connected components. 

The \textbf{complete affine flag variety} $Fl(V)$ is the space of all chains
of lattices $L_\bullet = (\cdots \subset L_1\subset L_2\subset \cdots \subset L_n\subset \cdots )$
such that $L_{i+n} = tL_i$ (equivalently, $t^{-1}L_{i+n} = L_i$)
and $\dim(L_i/L_{i-1}) = 1$
for all $i$. 
The \textbf{standard flag} is $E_\bullet := (\cdots \subset E_1\subset\cdots\subset E_n\subset \cdots )$,
whose stabilizer $\I$ is the subgroup of $G(\mathcal{O})$ which are
upper-triangular modulo $t^{-1}$:
\[\I = \{b=(b_{ij})\in G(\O): \ord(b_{ij})>0\  \forall i>j\}.\]
 Therefore $Fl(V) \cong G(\K)/\I$, and $\I$ is the \textbf{standard Iwahori
 subgroup} of $G(\K)$.
$Fl(V)$ decomposes as $\coprod_{j}Fl_j(V)$
where $FL_j(V):=G_j\cdot E_\bullet\cong G_j/\I$. Define also 
\[\I_-:=\{b=(b_{ij})\in GL_n(\C[t]): \ord(b_{ij})<0\  \forall i<j\},\]
(cf. \cite[Section 7.1]{galashin2022regularity}). For all $j\in\Z$, $Fl_j$ are isomorphic as ind-varieties via shifting by $\delta$, and the isomorphisms are equivariant with respect to the actions of $\I$ and $\I_-$. 

We define a map $\phi: G(\K)\to Mat_{\infty\times\infty}(\C)$ as follows. Given
$M = \sum_{i\le N} M_i t^{i}$ where $M_i\in Mat_{n\times n}(\C)$, we let
$\phi(M)_{ni+c, nj+d} = (M_{i-j})_{c,d}$. Namely, $\phi(M)$
is the following matrix:
\[
\begin{bmatrix}
\ddots &\vdots & \vdots & \vdots & \vdots&\reflectbox{$\ddots$} \\ 
\cdots & M_0 & M_{-1} & M_{-2} & M_{-3} &\cdots \\
\cdots & M_1 & M_{0} & M_{-1} & M_{-2} &\cdots \\
\cdots & \vdots & M_{1} & M_{0} & M_{-1} &\cdots \\
0 & M_N & \vdots & M_{1} & M_{0} &\cdots \\
0 & 0   & M_N    &\vdots & M_1 &\cdots \\
0 & 0   & 0      & M_N & \vdots & \cdots \\
\reflectbox{$\ddots$} & \vdots &\vdots &\vdots & \vdots & \ddots \\
\end{bmatrix}
\]
This presentation of  elements of $G(\mathcal{K})$ can be found in
\cite{Lam_2012} where it is used to define and study total
nonnegativity in loops groups.

Given a matrix in the image of $\phi$, one can construct a flag
$L_\bullet$ such that $L_i$ is the column span of the columns with indices
$\le i$. The standard flag is given by any element in $\I$, and in particular the $\infty\times \infty$
identity matrix. The Iwahori subgroup $\I$ consists of all elements in $G(\K)$ whose images are infinite upper-triangular matrices under $\phi$.

The \textbf{extended affine Weyl group} $\widetilde{W}$ for
$G(\mathcal{K})$ can be realized as the subset
$\{\pi:\Z\to\Z \text{ bijection}: \pi\delta^n=\delta^n\pi \}$
of the set of bijections on $\Z$, where $\delta(i)=i+1$. Hence for any
$w\in\widetilde{W}$ and $i\in \Z$, $w(i+n)=w(i)+n$. To specify an element in $w\in \widetilde{W}$, we use the notation $w=[w(1),\cdots, w(n)]$, as the rest of the values of $w$ are determined by periodicity.
Given any $w\in \widetilde{W}$, we let $w$ act $\mathcal{K}$-linearly on $V$
by $w (e_i) = e_{w(i)}$. This gives an 
embedding $\widetilde{W}\subset G(\mathcal{K})$.
The corresponding
\textbf{affine permutation matrix} is the $\infty\times\infty$
matrix $(a_{ij})$ where $a_{ij}=1$ if $i=w(j)$ and $a_{ij}=0$ otherwise. Note that this means we read off the permutation from a permutation matrix by \emph{columns} instead of rows.
We denote the coordinate flag determined by the affine permutation matrix of $w\in \widetilde{W}$ by $E^w_\bullet$.
 We also define
\[ W_j:=\widetilde{W}\cap G_j =\left\{w\in\widetilde{W}:\sum_{i=1}^n w(i)-i=nj\right\}. \]
For $0\le i\le n-1$, let $s_i\in W_0$ be the simple reflection given by
$s_i(i)=i+1$, $s_i(i+1)=i$, and $s_i(j)=j$
for all $j\not\equiv i \text{ or }i+1 \pmod n$. 
$W_0$  is the subgroup of $\widetilde{W}$
generated by $s_0,\cdots, s_{n-1}$. 
We have a semidirect product
$\widetilde{W}\cong \langle \delta\rangle \ltimes W_0$, where $\delta $ acts on
$W_0$ via the outer automorphism
$\delta s_i\delta^{-1} = s_{i+1}$. For $w\in\widetilde{W}$, we say $w$ has \textbf{index $j$} if
$w\in W_j$. Equivalently, $w=w'\delta^j$ for some $w'\in W_0$.

 For the reader's convenience, we review some standard results about  Bruhat order on $W_0$, following \cite[Chapter 8.3]{BB05}. Since $W_0$ is a Coxeter group, it has a length function $l:W_0\to \mathbb{N}$ that can be computed by counting \textbf{affine inversions} in the window notation of $w$. Concretely, we have
    \[
    l(w)=\mathrm{inv}(w(1),\ldots , w(n))+\sum_{1\leq k < \ell \leq n} \left\lfloor \frac{|w(k)-w(\ell)|}{n} \right\rfloor,
    \]
    where $\mathrm{inv}(w(1),\ldots , w(n))=|\{ (k,\ell) \;|\; 1\leq k < \ell \leq n, w(k)>w(\ell) \}|$ is the number of ordinary inversions.

    The \textbf{right descent set} $\operatorname{Des}(w)$ of an element $w$ of $W_0$ will be of  interest to us. It is defined as
    \[
    \operatorname{Des}(w):=\left\{ s_i : i\in \{0,1,\ldots ,n-1\}, ws_i<w \right\}.
    \]
    In the window notation $w=[w(1),\cdots ,w(n)]$, the simple reflection $s_i$ is a right descent if $w(i)>w(i+1)$ for $i=1,\ldots, n-1$, in particular $s_0$ is a right descent if $w(0)=w(n)-n>w(1)$. 

    In the extended affine Weyl group $\widetilde{W}$, each $W_j$ inherits an order from $W_0$. In other words, for any $w_1,w_2\in\widetilde W$, $w_1\le w_2$ if and only if $w_1=w_1'\delta^j$, $w_2=w_2'\delta^j$ for some $w_1',w_2'\in W_0$ and $j\in \Z$, and $w_1'\le w_2'$ in Bruhat order.

  $G(\K)$ admits a Bruhat decomposition, $G(\K)=\bigsqcup_{w\in \widetilde{W}} \I w\I$.  Hence we also have a
Bruhat decomposition
of the complete affine flag variety,
$Fl(V) = \bigsqcup_{w\in \widetilde{W}} \mathcal{X}^w_\circ$ where $\mathcal{X}^w_\circ := \I w\I/\I=\I wE_\bullet$.
As remarked before, since $Fl(V)$ decomposes into $\Z$-many isomorphic components, we will now restrict attention to  $Fl_0$, and the analysis on the other copies are essentially identical. Define for any $w\in W_0$ \[ \mathcal{X}^w_\circ := \I w\I/\I=\I wE_\bullet\text{ and }\mathcal{X}_w^\circ := \I_-w\I/\I=\I_- wE_\bullet \footnote{We use superscripts instead of the usual subscripts for Schubert cells/varieties. One motivation for this notation is that, in the usual type $A$ setting, taking the $B$-orbits of coordinate flags corresponds to \emph{upward} row operations. Another reason for this is that with this convention, $\mathcal{X}^w_v$ is nonempty if $v\leq w$.},\]
we have 
\[\mathcal{X}^w:=\overline{\mathcal{X}^w_\circ}=\bigsqcup_{v\le w}\mathcal{X}^v_\circ\text{ and }\mathcal{X}_w:=\overline{\mathcal{X}_w^\circ}=\bigsqcup_{u\ge w}\mathcal{X}_u^\circ.\]
We call  $\mathcal{X}^w_\circ$  a \textbf{Schubert cell} and  $\mathcal{X}_w^\circ$  an \textbf{opposite Schubert cell}, and similarly  $\mathcal{X}^w$  a \textbf{Schubert variety} and $\mathcal{X}_w$  an \textbf{opposite Schubert variety}.
We emphasize that each $\mathcal{X}^w$ is a finite-dimensional variety, whereas each $\mathcal{X}_w$ is an infinite-dimensional ind-variety. In particular, $\mathcal{X}^\circ_\mathrm{id}$ is open inside $Fl_0(V)$. 

Explicitly, $\mathcal{X}^w_\circ$ and $\mathcal{X}_w^\circ$ and their closures can be  described in terms of incidence conditions via lattices.
Suppose $w\in W_0$, we have \[\mathcal{X}^w_\circ = \{L_\bullet\in Fl_0(V):\forall i,j,\ \dim (E_i/(L_j\cap E_i)) = |\Z_{\le i} - w\Z_{\le j}|\}. \] 
Here $w\Z_{\le i}=\{w(k):k\le i\}$. Replacing the equality with
$\le$ we get conditions for defining 
$\mathcal{X}^w = \overline{\mathcal{X}^w_\circ}$. 

Let $E^j=\mathrm{fSpan}_\C\{e_i\}_{i> j}$ denote the $j$th 
\textbf{standard anti-lattice}. Notice that $E^\bullet=(\cdots \supset E^j \supset E^{j+1}\supset \cdots)$ is not an element
in the affine flag variety.
For each $v\in W_0$, we have
\[\mathcal{X}_v =\{L_\bullet\in Fl_0(V):\forall\ i,j,\ \dim(E^i \cap L_j)\ge |v\Z_{\le j}\cap \Z_{> i}|\}. \]


 Define the \textbf{Kazhdan-Lusztig} variety
 $\mathcal{X}_{\circ,v}^w:=\mathcal{X}^w_\circ \cap \mathcal{X}_v$.  Then $\mathcal{X}_{\circ,v}^w$ is nonempty if and only if $v\le w$. This follows from \cite[Lemma 7.1.22(b)]{kumar2012kac}.

\begin{rmk}
    In \cite{galashin2022regularity}, the affine flag variety was realized using the polynomial loop group, $GL_n(\mathbb{C}[t,t^{-1}])$. We note that the affine flag variety defined this way is isomorphic to the one defined by our definition; see \cite[Remark A.3]{galashin2022regularity}. The analysis done in this paper would not change significantly if we also used the polynomial loop group setting. 
\end{rmk}

Finally, for convenience, we define for any $a\in \mathbb{Z}$ the notation $[a]_n$ for the set of integers congruent to $a \pmod n$.

\section{Equations for opposite Schubert conditions}
In this section, let $M\in G_0\subset G(\mathcal{K})$. 
Then $M$ defines an affine flag
 $L_\bullet=[M]\in G(\K)/\I$, where
$L_j$ is the span of the columns of $\phi(M)$
with indices no greater than
$j$. For $I,J\subset \Z$, we let $M_{I,J}$ denote
the submatrix of $\phi(M)$ obtained by taking
the entries whose row indices are in $I$
and column indices are in $J$.
Furthermore,
for $j\in \Z$, we let $M_{*,j}$ denote the $j$th
column of $\phi(M)$.

\begin{defn}
    Let $v\in W_0$. 
    The \textbf{ diagram of $v$} is the collection of boxes
    \[
        D(v):=\{ (i,j) \; | \; v^{-1}(i)>j, i<v(j) \}.
    \]

    Alternatively, for each nonzero entry $(v(j),j)$, cross out all boxes below and to the right, the remaining boxes constitute $D(v)$.
\end{defn}
\begin{defn}
Let $v\in W_0$. The \textbf{essential set} of $v$
is 
\[\ess(v):=\{(i,j)\in\mathbb{Z}\times\mathbb{Z}:v(j)>i, v^{-1}(i)>j, v(j+1)\le i, v^{-1}(i+1)\le j\}.\]
\end{defn}
Diagrammatically, $\ess(v)$ consists of positions of cells that are southeast corners of $D(v)$.
\begin{defn}
For any $(i,j)\in\Z\times\Z$ and $v\in W_0$,
let \[l_v(i,j):=\max(0,j-\min_{k>i} (v^{-1}(k))+1)\] 
and for any $l\ge l_v(i,j)$, \[n_v(i,j,l):=|\{j':j-l<j'\le j, v(j')\le i\}|.\]
\end{defn}
In other words, $l_v(i,j)$ is the smallest $l$ such that all columns $\le j-l$ contain a 1 weakly north of row $i$ (i.e., $v(j')\le i$ for all $j'\le j-l$), and $n_v(i,j,l)$ counts the number of 1's in rows $\le i$ and columns $(j-l,j]$ for $l\ge l_v(i,j)$.

\begin{example}
We show an exmaple in Figure~\ref{fig:window}. Let $v=[1,5,0]$ and $(i,j)=(0,2)$. Here $l_v(i,j)=4$ and $n_v(i,j,4)=1$. For all $a\ge 0$, $n_v(i,j,4+a)=1+a$. The essential boxes in this diagram are denoted by $e$.
\end{example}
\begin{figure}
\begin{tikzpicture}[scale=0.8,transform shape]
\pgfmathsetmacro{\wid}{0.5}
\draw[LightSkyBlue1, step=\wid] (0,0) grid (7,7);
    \node[font=\tiny] at (0.5*\wid, 7.2){$\cdots$};
    \foreach  \j  in {2,...,13} {
       \pgfmathtruncatemacro{\x}{\j-8}%
       \node[font=\tiny] at (\j*\wid-0.5*\wid,7.2) {$\x$};
    };
    \node[font=\tiny] at (0.5*\wid+6.5, 7.2){$\cdots$};

    \node[font=\tiny] at (-0.3, 0.5*\wid){$\vdots$};
    \foreach  \j  in {2,...,13} {
       \pgfmathtruncatemacro{\x}{8-\j}%
       \node[font=\tiny] at (-0.3,\j*\wid-0.5*\wid) {$\x$};
    };
    \node[font=\tiny] at (-0.3, 0.5*\wid+6.7){$\vdots$};
    \foreach \j in {0,...,3} {
        \pgfmathsetmacro{\x}{((\j*3)+4)*\wid+0.5*\wid}%
        \pgfmathsetmacro{\y}{(-(\j*3)+13)*\wid+0.5*\wid}%
        \node[] at (\x,\y) {$1$};
        \draw[gray] (\x+0.2*\wid,\y) -- (7,\y);
        \draw[gray] (\x, \y-0.3*\wid) -- (\x, 0);
    }
    \foreach \j in {0,...,3} {
        \pgfmathsetmacro{\x}{((\j*3)+2)*\wid+0.5*\wid}%
        \pgfmathsetmacro{\y}{(-(\j*3)+12)*\wid+0.5*\wid}%
        \node[] at (\x,\y) {$1$};
        \draw[gray] (\x+0.2*\wid,\y) -- (7,\y);
        \draw[gray] (\x, \y-0.3*\wid) -- (\x, 0);
    }
    \foreach \j in {0,...,3} {
        \pgfmathsetmacro{\x}{((\j*3))*\wid+0.5*\wid}%
        \pgfmathsetmacro{\y}{(-(\j*3)+12)*\wid+0.5*\wid}%
        \node[] at (\x,\y) {$e$};
    }
    \foreach \j in {0,...,3} {
        \pgfmathsetmacro{\x}{((\j*3)+3)*\wid+0.5*\wid}%
        \pgfmathsetmacro{\y}{(-(\j*3)+13)*\wid+0.5*\wid}%
        \node[] at (\x,\y) {$e$};
    }
    \foreach \j in {0,...,3} {
        \pgfmathsetmacro{\x}{((\j*3))*\wid+0.5*\wid}%
        \pgfmathsetmacro{\y}{(-(\j*3)+11)*\wid+0.5*\wid}%
        \node[] at (\x,\y) {$1$};
        \draw[gray] (\x+0.2*\wid,\y) -- (7,\y);
        \draw[gray] (\x, \y-0.3*\wid) -- (\x, 0);
    }
    \draw[gray] (1.5*\wid, 0) -- (1.5*\wid, 7);
    \draw[blue, thick] (0, 7*\wid) -- (14*\wid, 7*\wid);
    \draw[blue, thick] (10*\wid, 0) -- (10*\wid, 14*\wid);
    \draw[red, thick, dashed] (6*\wid, 0) -- (6*\wid, 14*\wid);
\end{tikzpicture}
\caption{Example for $l_v(i,j)$ and $n_v(i,j,l)$}
\label{fig:window}
\end{figure}

The following lemma shows a correspondence
between the dimension condition on 
each $E^i\cap L_j$ coming from $v$ and conditions
on certain minors of $\phi(M)$.
\begin{lemma}
\label{lem1}
For all $(i,j)\in \mathbb{Z}\times \mathbb{Z}$, 
$\dim(E^i\cap L_j)\ge |v\mathbb{Z}_{\le j}\cap \mathbb{Z}_{>i}|$
if and only if there exists
$l\ge l_v(i,j)$ such that all minors
$\det (M_{I,J})$ where $I\subset \mathbb{Z}_{\le i}$, $J\subseteq [j-l+1, j]$,
and
$|I|=|J| =n_v(i,j,l)+1$ vanish.
\end{lemma}
\begin{proof}
Suppose $\dim(E^i\cap L_j)\ge |v\mathbb{Z}_{\le j}\cap \mathbb{Z}_{>i}|$.
In other words, there exists $l>0$
such that 
$\dim(\text{span}\{M_{*,j-l+1},\cdots, M_{*,j}\}\cap E^i )\ge |v\mathbb{Z}_{\le j}\cap \mathbb{Z}_{>i}|$.   Suppose the rank of the matrix 
$M_{(-\infty, i],[j-l+1,j]}$  is $r$. This means that
we can perform column operations on
$M_{(-\infty, \infty),[j-l+1,j]}$  such that there are exactly
$l-r$ columns whose entries with row
indices less than or equal to $i$ are 0. Therefore, this is
equivalent to $\dim(\text{span}\{M_{*,j-l+1},\cdots M_{*,j}\}\cap E^i )=l-r\ge |v\mathbb{Z}_{\le j}\cap \mathbb{Z}_{>i}|$.  In other words, 
$r\le l-|v\mathbb{Z}_{\le j}\cap \mathbb{Z}_{>i}|.$ We may take $l\ge l_v(i,j)$ so that $l-|v\mathbb{Z}_{\le j}\cap \mathbb{Z}_{>i}|=n_v(i,j,l)$. 
This is true if and only if all minors of size
$n_v(i,j,l)+1$ in the matrix $M_{(-\infty, i],[j-l+1,l]}$ vanish.
\end{proof}

\begin{lemma}
\label{lem2}
Suppose $v\in W_0$. For any $j\in\Z$,
there exists $N<\!< j$ such that for all $i<N$,
$|v\Z_{\le j}\cap \Z_{>i}|=j-i$.
\end{lemma}
\begin{proof}
When $v$ is the identity the claim is clear. Suppose the claim is true for $v$; we proceed by induction to show it's also true for $sv>v$, where $s$ is a simple reflection. 
Let $j\in \Z$. Pick $N<\!<j$ such that the claim is true by induction hypothesis. 
Since $v$ is an affine permutation, we may further assume that $N$ is small enough so that for any $i< N$, $v^{-1}(i)<j$ and $sv^{-1}(i)<j$. 
Pick any $i<N$.
In the permutation matrix of $v$,
$|v\Z_{\le j}\cap \Z_{>i}|$ counts the number of $1$'s with row indices $>i$ an column indices $<j$. (For example in Figure~\ref{fig:window}, if $j=2$ and $i=-2$, $|v\Z_{\le j}\cap \Z_{>i}|=4=j-i$.) Consider multiplying $v$ on the left by
$s$, which swaps corresponding adjacent 
rows. Now notice the following for any
$j'\le j$:
\begin{itemize}
\item If $v(j')>i+1$, then $sv(j')>i$.
Therefore
$v(j')\in v\Z_{\le j}\cap \Z_{>i}$
and $sv(j')\in sv\Z_{\le j}\cap \Z_{>i}$.
\item If $v(j')<i$, then $sv(j')\le i$.
Therefore
$v(j')\not\in v\Z_{\le j}\cap \Z_{>i}$
and $sv(j')\not\in sv\Z_{\le j}\cap \Z_{>i}$.
\item
If $v(j')=i+1$ and $sv(j')=i+2$, we have
$v(j')\in v\Z_{\le j}\cap \Z_{>i}$
and $sv(j')\in sv\Z_{\le j}\cap \Z_{>i}$.
\item 
If $v(j')=i$ and $sv(j')=i-1$, we have
$v(j')\not\in v\Z_{\le j}\cap \Z_{>i}$
and $sv(j')\not\in sv\Z_{\le j}\cap \Z_{>i}$.
\item
If $v(j')=i+1$ and $sv(j')=i$,
then $j'':=v^{-1}(i)<j$ and $sv(j'')=i+1$.
Namely, $v(j')\in v\Z_{\le j}\cap \Z_{>i}$,
$sv(j')\not\in sv\Z_{\le j}\cap \Z_{>i}$,
$v(j'')\not\in v\Z_{\le j}\cap \Z_{>i}$,
and $sv(j'')\in sv\Z_{\le j}\cap \Z_{>i}$.
\item If $v(j')=i$ and $sv(j')=i+1$,
then $j'':=v^{-1}(i+1)<j$ and $sv(j'')=i$. 
Namely, $v(j')\not\in v\Z_{\le j}\cap \Z_{>i}$,
$sv(j')\in sv\Z_{\le j}\cap \Z_{>i}$,
$v(j'')\in v\Z_{\le j}\cap \Z_{>i}$,
and $sv(j'')\not\in sv\Z_{\le j}\cap \Z_{>i}$.
\end{itemize}
The above argument shows that there is a bijection
between the sets $\{j':j'\le j, v(j')>i\}$ and
$\{j': j'\le j, sv(j')>i\}$ for any simple reflection $s$.
The claim then follows by induction.
\end{proof}

The following Lemma proves that it is enough to check the conditions coming from essential boxes.

\begin{lemma}
\label{lem3}
Let $M\in G_0$ and $L_\bullet=[M]\in Fl_0(V)$. 
Then for any $v\in W_0$, 
if for all  $(i,j)\in\ess(v)$, 
$\dim(E^i\cap L_j)\ge |v\mathbb{Z}_{\le j}\cap \mathbb{Z}_{>i}|$, we have
 for all  $(i,j)\in \Z\times \Z$,
$\dim(E^i\cap L_j)\ge |v\mathbb{Z}_{\le j}\cap \mathbb{Z}_{>i}|$.
\end{lemma}
\begin{proof}
\textbf{Case 1} (when $(i,j)$ lies in a region that is not crossed out). 
Suppose $v^{-1}(i)>j$ and $v(j)>i$. 
Then $(i,j)$ lies in the same connected component
as some $(i',j')\in\ess(v)$ for which $i'\ge i$
and $j'\ge j$. 
 Suppose 
$v(j+1)>i$, then $|v\mathbb{Z}_{\le j+1}\cap \mathbb{Z}_{>i}|=|v\mathbb{Z}_{\le j}\cap \mathbb{Z}_{>i}|+1$. 
Therefore we have the statement
\[\dim(E^i\cap L_{j+1})\ge |v\Z_{\le j+1}\cap \Z_{>i}|\implies \dim(E^i\cap L_j)\ge |v\mathbb{Z}_{\le j}\cap \mathbb{Z}_{>i}|.\] 
Similarly suppose
$v^{-1}(i+1)>j$, then $|v\mathbb{Z}_{\le j}\cap \mathbb{Z}_{>i+1}|=|v\mathbb{Z}_{\le j}\cap \mathbb{Z}_{>i}|$.
Therefore we have
the statement \[\dim (E^{i+1}\cap L_j)\ge |v\mathbb{Z}_{\le j}\cap \mathbb{Z}_{>i+1}|\implies \dim (E^i\cap L_j)\ge |v\mathbb{Z}_{\le j}\cap \mathbb{Z}_{>i}|.\] 
These two statements combined give that  $\dim(E^{i'}\cap L_{j'})\ge |v\mathbb{Z}_{\le j'}\cap \mathbb{Z}_{>i'}|$ implies $\dim(E^i\cap L_j)\ge |v\mathbb{Z}_{\le j}\cap \mathbb{Z}_{>i}|$. 

\textbf{Case 2} (when $(i,j)$ lies in a region that is crossed out
by a vertical ray but not a horizontal ray). Suppose 
$v(j)\le i$ and $v^{-1}(i)>j$. In this case
either $v(j-1)\le i$ or $v(j-1)>i$ . In other words, 
if an entry is crossed out only vertically, the entry to the left
is either also crossed out only vertically, or not crossed out.
Notice that $v(j)\le i$  implies $|v\Z_{\le j}\cap \Z_{>i}|=|v\Z_{\le j-1}\cap \Z_{>i}|$. Therefore we
have the statement
\[\dim(E^i\cap L_{j-1})\ge|v\Z_{\le j-1}\cap \Z_{>i}|\implies
\dim(E^i\cap L_{j})\ge|v\Z_{\le j}\cap \Z_{>i}|.\]
Thus if there exists $j'<j$ such that $v(j')>i$, by Case 1
we know $\dim(E^i\cap L_{j'})\ge |v\Z_{\le j'}\cap \Z_{>i}|$. By picking the 
largest such $j'$ and
(repeatedly) applying the
statement above, we get $\dim(E^i\cap L_{j})\ge |v\Z_{\le j}\cap \Z_{>i}|$. Otherwise if 
$v(j')\le i$ for all $j'<j$, all entries
to the left of $(i,j)$ are crossed out (only) vertically. 
However in this case we know that $|v\Z_{\le j'}\cap \Z_{>i}|=0$ for all $j' < j$. 
Using the statement above again  
we may deduce $\dim(E^i\cap L_{j})\ge|v\Z_{\le j}\cap \Z_{>i}|$
as desired.

\textbf{Case 3} (when $(i,j)$ lies in a region that is crossed out
by a horizontal ray but not a vertical ray). Suppose
$v(j)>i$ and $v^{-1}(i)\le j$. In this case 
$v^{-1}(i-1)\le j$ or $v^{-1}(i-1)>j$. In other words, 
if an entry is crossed out only horizontally, the entry immediately
above is either also crossed out only horizontally, or not 
crossed out. Notice that 
$v^{-1}(i)\le j$ implies $|v\Z_{\le j}\cap \Z_{>i-1}|=|v\Z_{\le j}\cap \Z_{>i}|+1$. 
Therefore we have the
statement 
\[\dim(E^{i-1}\cap L_{j})\ge|v\Z_{\le j}\cap \Z_{>i-1}|\implies
\dim(E^{i}\cap L_{j})\ge|v\Z_{\le j}\cap \Z_{>i}|.\] 
Thus if there exists $i'<i$ such
that $v^{-1}(i')>j$, by Case 1 we know that
$\dim(E^{i'}\cap L_{j})\ge|v\Z_{\le j}\cap \Z_{>i'}|$.
By picking the largest such $i'$ and 
 (repeatedly) applying the statement above, we get
 $\dim(E^{i}\cap L_{j})\ge|v\Z_{\le j}\cap \Z_{>i}|$. Otherwise if 
$v^{-1}(i') \le j$ for all $i'<i$, all entries
directly above $(i,j)$ are crossed out only horizontally.
We know by 
Lemma \ref{lem2}
that there exists $i'<\!<i$  sufficiently small 
such that $|v\Z_{\le j}\cap \Z_{>i'}|= j-i'=\dim (E^{i'}\cap E_j)$.
Since $(L_\bullet)\in Fl_0(V)$, we know
$\dim (E^{i'}\cap L_j)\ge \dim (E^{i'}\cap E_j)$, 
so $\dim (E^{i'}\cap L_j)\ge |v\Z_{\le j}\cap \Z_{>i'}|$.
We may then deduce that  $\dim(E^{i}\cap L_{j})\ge|v\Z_{\le j}\cap \Z_{>i}|$ by the
statement above.

\textbf{Case 4} (when $(i,j)$ is crossed out by both a vertical and a horizontal ray). Suppose $v(j)\le i$
and $v^{-1}(i)\le j$. Similar to the argument in Case 2, we have
\[\dim(E^i\cap L_{j-1})\ge |v\Z_{\le j-1}\cap \Z_{>i}|\implies \dim(E^i\cap L_j)\ge |v\Z_{\le j}\cap \Z_{>i}|.\]Let  $j'=v^{-1}(i)-1$. Then $(i,j')$
is either not crossed out as in Case 1 or is only crossed
out vertically as in Case 2. 
These together with the statements above give the desired result.
\end{proof}
\begin{thm}
\label{thm:seteq}
Let $M\in G_0$ and $v\in W_0$.
$[M]\in\mathcal{X}_v\subset Fl_0(V)$
if and only if for all
$(i,j)\in \ess(v)\subset{\Z\times \Z}$,
there exists $l\ge l_v(i,j)$
such that such that all minors of size
$n_v(i,j,l)+1$ in the submatrix $\phi(M)_{(-\infty, i],[j-l+1,j]}$
of $\phi(M)$ vanish.
\end{thm}
\begin{proof}
By Lemma \ref{lem3} we know
the dimension conditions
coming from the $\ess(v)$
imply  dimension conditions
for all $(i,j)$, and Lemma
\ref{lem1} translates the
dimension  conditions
to conditions on the matrix
$\phi(M)$. Thus our theorem
is a direct consequence of these
two lemmas.
\end{proof}

\section{Equations for Kazhdan-Lusztig vareties in $Fl_0(V)$}
\subsection{A preferred Bott-Samelson parametrization for $\mathcal{X}_\circ^w$}
\label{subsec:param}

For $\alpha=1,\cdots,n-1$, let $m_\alpha(x)=(m_{ij})\in G(\mathcal{K})$ denote the $n$ by $n$ matrix such that
$m_{\alpha,\alpha}=x$, $m_{\alpha,\alpha+1}=1$, $m_{\alpha+1,\alpha}=1$, $m_{\alpha+1,\alpha+1}=0$, and
$m_{i,j}$ identical to the identity matrix in all other entries. For $\alpha=0$, let $m_\alpha(x)=(m_{ij})\in G(\mathcal{K})$ denote the matrix such that $m_{1,1}=0$, $m_{1,n}=t$, $m_{n,1}=t^{-1}$, and $m_{n,n}=x$, and $m_{i,j}$ identical to the identity matrix in all other entries. In other words, for $\alpha=0,\cdots,n-1$, $m_\alpha(x)=\exp(xR_\alpha){s_\alpha}$, where $R_\alpha$ is a root vector that spans the root space indexed by $\alpha$. For example, when $n=3$, we have
\[
m_0(x)=\begin{bmatrix}
   0 & 0 & t \\
   0 & 1 & 0 \\
   t^{-1} & 0 & x
\end{bmatrix},
m_1(x)=\begin{bmatrix}
   x & 1 & 0 \\
   1 & 0 & 0 \\
   0 & 0 & 1
\end{bmatrix},
m_2(x)=\begin{bmatrix}
   1 & 0 & 0 \\
   0 & x & 1 \\
   0 & 1 & 0
\end{bmatrix}.
\]
For any $w\in W_0$, let $Q=(s_{\alpha_1},\cdots ,s_{\alpha_{\ell(w)}})$ be a reduced word for $w$. 
Define the \textbf{Bott-Samelson map}
$m_Q:\mathbb{C}^{\ell(w)}\to G(\mathcal{K})$ such that
\[m_Q(x_1,\cdots,x_{\ell(w)})=\prod_{i=1}^{\ell(w)}m_{\alpha_i}(x_i).\]
This map descends to an isomorphism between $\mathbb{C}^{\ell(w)}$ and 
$\mathcal{X}^w_\circ$ and hence gives a parametrization of $\mathcal{X}^w_\circ$, see e.g., \cite[Chapter VII]{kumar2012kac}.

Typically for an arbitrary choice of a reduced word
$Q$, the 
entries in $\phi(m_Q(\mathbf{x}))$ are not linear expressions of the variables.
We give an inductive recipe for a choice of reduced word for $w$ such that each entry in $\phi(m_Q(\mathbf{x}))$ is 0, 1, or $x_i$ for some $i$.


Let $i$ denote the smallest row index in $D(w)$ such that there exists $1\le j\le n$ for which $(i,j)$ is uncrossed. Such $(i,j)$ with $j$ maximal will be called the \textbf{preferred box of $w$}. Let $\sigma(w):=j \bmod n$. Note that $i$ must exist since $w$ is an affine permutation: for all $i'$ sufficiently small, $w^{-1}(i')<0$, so $(i',j')$ with $1\le j' \le n$ are all crossed out. In terms of our window notation, the preferred box is $(i,j)$ such that $w(j+1)$ is minimal among $j\in \operatorname{Des}(w)$.
 \begin{lemma}
 \label{lem:preferred}
 If $(i,j)$ is the preferred box of $w$,
 $(i,j+1)$ must be a 1 in the permutation matrix of $w$.
 \end{lemma}

 \begin{proof}
 We prove this by contradiction. 
\begin{enumerate} 
\item If $(i,j+1)$ is an uncrossed box and $j<n$ this contradicts the maximality of $j$.
\item If $(i,j+1)$ is an uncrossed box and $j=n$, then since $w$ is an affine permutation, $(i-n,j-n+1)$ is also an uncrossed box and this contradicts the minimality of $i$. 
\item If $(i,j+1)$ is crossed out because of a $1$ in position $(w(j+1),j+1)$ with $w(j+1)<i$, then we claim that $(w(j+1),j)\in D(w)$. Since $(i,j)$ is not crossed out, $w(j)>i$ and in turn $w(j)>w(j+1)$. So $(w(j+1),j)\in D(w)$, and as $w(j+1)<i$, this contradicts the minimality of $i$.
\item If $(i,j+1)$ is crossed out because of a $1$ in position $(i,h)$ with $h<j$, then $(i,j)$ must also be crossed out, a contradiction.
\end{enumerate}
\end{proof}
Note that $ws_{\sigma(w)}<w$, since $i=w(j+1)<w(j)$. We are now ready to introduce the preferred reduced word for $w$ that gives a Bott-Samelson parametrization for $\mathcal{X}^w_\circ$.
Define inductively \[Q(w):=
\begin{cases}
Q(ws_{\sigma(w)})s_{\sigma(w)} &\text{ if } w\neq \id \\
\text{the empty word} &\text{ if } w=\id.
\end{cases}\]
We call $Q(w)$ the \textbf{preferred word} for $w$.

\begin{prop}
\label{prop:linearBS}
The Bott-Samelson map $m_{Q(w)}$ gives a linear parametrization of $\mathcal{X}^w_\circ$. 
In particular, each entry in $M:=\phi(m_{Q(w)}(\mathbf{x}))$ is 0, 1, or $x_i$ for some $i$.

\end{prop}
\begin{proof}
We first note that for any $M\in G(\mathcal{K})$, the effect of multiplying $M$ on the right by $m_{\alpha}(x)$ seen on $\phi(M)$ is the composition of the following two operations:
\begin{enumerate}[(1)]
    \item add $x$ times column $b$ to column $b+1$, for each $b\equiv \alpha \pmod n$;
    \item swap columns $b$ and $b+1$, for each $b\equiv \alpha \pmod n$.
\end{enumerate}

Furthermore, it is easy to see by induction that for any reduced word $Q$ of $w$, $\phi(m_Q(\mathbf{x}))$ has 1's in the same positions as in the permutation matrix of $w$, 0's in the crossed-out entries as the Rothe diagram of $w$, and polynomials of $x_i$'s in positive degrees in entries of $D(w)$.

We assume by induction that the proposition is true for $Q(ws_{\sigma(w)})$.  Let $M:=m_{Q(ws_{\sigma(w)})}(x_1,\cdots,x_{\ell(w)-1})$ and $M':=m_{Q(w)}(x_1,\cdots,x_{\ell(w)})$. Let $(i,j)$ be the preferred box of $w$. Recall that $j\in [\sigma(w)]_n$.
Then $\phi(M')$ can be obtained from $\phi(M)$ by adding $x_{\ell(w)}$ times column $b$ to column $b+1$ for each $b\in [\sigma(w)]_n$, and then swapping columns $b$ and $b+1$ for each $b\in [\sigma(w)]_n$. By choice of $(i,j)$ and Lemma~\ref{lem:preferred}, for every $b\in[\sigma(w)]_n$, column $b+1$ in $\phi(M')$ contains a single 1. (In particular, there are no other nonzero entries above this 1 because $i$ is minimal by the definition of preferred box.) Therefore, for every $b\in [\sigma(w)]_n$, column $b$ in $\phi(M)$ contains a single 1, and therefore all entries strictly to the right or below this one in $\phi(M)$ are all 0. The claim then easily follows.
\end{proof}
\begin{example}
\label{ex:chart}
Let $w=[-5,5,0,10]$. Then one can compute \[Q(w)=s_1s_0s_3s_1s_0s_2s_3s_1s_2s_3s_0.\]
The Bott-Samelson matrix for this preferred word 
is illustrated by the top matrix in Figure~\ref{fig:bigex}. (The two ``windows'' in the picture will be explained later.)
\begin{figure}
    \centering
\begin{tikzpicture}[scale=0.8,transform shape]
\pgfmathsetmacro{\wid}{0.5}

\draw[LightSkyBlue1, step=\wid] (0,0) grid (12,12);
    \node[font=\tiny] at (0.5*\wid, 12.2){$\cdots$};
    \foreach  \j  in {2,...,23} {
       \pgfmathtruncatemacro{\x}{\j-10}%
       \node[font=\tiny] at (\j*\wid-0.5*\wid,12.2) {$\x$};
    };
    \node[font=\tiny] at (0.5*\wid+11.5, 12.2){$\cdots$};

    \node[font=\tiny] at (-0.3, 0.5*\wid){$\vdots$};
    \foreach  \j  in {2,...,23} {
       \pgfmathtruncatemacro{\x}{15-\j}%
       \node[font=\tiny] at (-0.3,\j*\wid-0.5*\wid) {$\x$};
    };
    \node[font=\tiny] at (-0.3, 0.5*\wid+11.5){$\vdots$};
    \foreach \j in {0,...,5} {
        \pgfmathsetmacro{\x}{((\j*4)+3)*\wid+0.5*\wid}%
        \pgfmathsetmacro{\y}{(-(\j*4)+23)*\wid+0.5*\wid}%
        \node[] at (\x,\y) {$x_9$};
    }
    \foreach \j in {0,...,4} {
        \pgfmathsetmacro{\x}{((\j*4)+4)*\wid+0.5*\wid}%
        \pgfmathsetmacro{\y}{(-(\j*4)+23)*\wid+0.5*\wid}%
        \node[] at (\x,\y) {$x_{10}$};
    }
    \foreach \j in {0,...,4} {
        \pgfmathsetmacro{\x}{((\j*4)+5)*\wid+0.5*\wid}%
        \pgfmathsetmacro{\y}{(-(\j*4)+23)*\wid+0.5*\wid}%
        \node[] at (\x,\y) {$x_{11}$};
    }
    \foreach \j in {0,...,4} {
        \pgfmathsetmacro{\x}{((\j*4)+6)*\wid+0.5*\wid}%
        \pgfmathsetmacro{\y}{(-(\j*4)+23)*\wid+0.5*\wid}%
        \node[] at (\x,\y) {$1$};
    }
    \foreach \j in {0,...,5} {
        \pgfmathsetmacro{\x}{((\j*4)+3)*\wid+0.5*\wid}%
        \pgfmathsetmacro{\y}{(-(\j*4)+22)*\wid+0.5*\wid}%
        \node[] at (\x,\y) {$x_7$};
    }
    \foreach \j in {0,...,4} {
        \pgfmathsetmacro{\x}{((\j*4)+4)*\wid+0.5*\wid}%
        \pgfmathsetmacro{\y}{(-(\j*4)+22)*\wid+0.5*\wid}%
        \node[] at (\x,\y) {$1$};
    }
    \foreach \j in {0,...,4} {
        \pgfmathsetmacro{\x}{((\j*4)+3)*\wid+0.5*\wid}%
        \pgfmathsetmacro{\y}{(-(\j*4)+19)*\wid+0.5*\wid}%
        \node[] at (\x,\y) {$x_5$};
    }
    \foreach \j in {-1,...,4} {
        \pgfmathsetmacro{\x}{((\j*4)+5)*\wid+0.5*\wid}%
        \pgfmathsetmacro{\y}{(-(\j*4)+19)*\wid+0.5*\wid}%
        \node[] at (\x,\y) {$x_8$};
    }
    \foreach \j in {0,...,4} {
        \pgfmathsetmacro{\x}{((\j*4)+3)*\wid+0.5*\wid}%
        \pgfmathsetmacro{\y}{(-(\j*4)+18)*\wid+0.5*\wid}%
        \node[] at (\x,\y) {$x_4$};
    }
    \foreach \j in {-1,...,4} {
        \pgfmathsetmacro{\x}{((\j*4)+5)*\wid+0.5*\wid}%
        \pgfmathsetmacro{\y}{(-(\j*4)+18)*\wid+0.5*\wid}%
        \node[] at (\x,\y) {$x_6$};
    }
    \foreach \j in {0,...,4} {
        \pgfmathsetmacro{\x}{((\j*4)+3)*\wid+0.5*\wid}%
        \pgfmathsetmacro{\y}{(-(\j*4)+17)*\wid+0.5*\wid}%
        \node[] at (\x,\y) {$1$};
    }
    \foreach \j in {-1,...,3} {
        \pgfmathsetmacro{\x}{((\j*4)+5)*\wid+0.5*\wid}%
        \pgfmathsetmacro{\y}{(-(\j*4)+15)*\wid+0.5*\wid}%
        \node[] at (\x,\y) {$x_3$};
    }
    \foreach \j in {-1,...,3} {
        \pgfmathsetmacro{\x}{((\j*4)+5)*\wid+0.5*\wid}%
        \pgfmathsetmacro{\y}{(-(\j*4)+14)*\wid+0.5*\wid}%
        \node[] at (\x,\y) {$x_2$};
    }
    \foreach \j in {-1,...,3} {
        \pgfmathsetmacro{\x}{((\j*4)+5)*\wid+0.5*\wid}%
        \pgfmathsetmacro{\y}{(-(\j*4)+13)*\wid+0.5*\wid}%
        \node[] at (\x,\y) {$x_1$};
    }
    \foreach \j in {-1,...,3} {
        \pgfmathsetmacro{\x}{((\j*4)+5)*\wid+0.5*\wid}%
        \pgfmathsetmacro{\y}{(-(\j*4)+12)*\wid+0.5*\wid}%
        \node[] at (\x,\y) {$1$};
    }
    \draw[blue, thick] (14*\wid, 11*\wid) -- (9*\wid, 11*\wid);
    \draw[blue, thick] (14*\wid, 11*\wid) -- (14*\wid, 24*\wid);
    \draw[blue, thick] (9*\wid, 11*\wid) -- (9*\wid, 24*\wid);
    
    \draw[red, thick] (13*\wid, 9*\wid) -- (9.1*\wid, 9*\wid);
    \draw[red, thick] (13*\wid, 9*\wid) -- (13*\wid, 24*\wid);
    \draw[red, thick] (9.1*\wid, 9*\wid) -- (9.1*\wid, 24*\wid);
\end{tikzpicture}

\begin{tikzpicture}[scale=0.8,transform shape]
\pgfmathsetmacro{\wid}{0.5}

\draw[LightSkyBlue1, step=\wid] (0,0) grid (12,12);
    \node[font=\tiny] at (0.5*\wid, 12.2){$\cdots$};
    \foreach  \j  in {2,...,23} {
       \pgfmathtruncatemacro{\x}{\j-10}%
       \node[font=\tiny] at (\j*\wid-0.5*\wid,12.2) {$\x$};
    };
    \node[font=\tiny] at (0.5*\wid+11.5, 12.2){$\cdots$};

    \node[font=\tiny] at (-0.3, 0.5*\wid){$\vdots$};
    \foreach  \j  in {2,...,23} {
       \pgfmathtruncatemacro{\x}{15-\j}%
       \node[font=\tiny] at (-0.3,\j*\wid-0.5*\wid) {$\x$};
    };
    \node[font=\tiny] at (-0.3, 0.5*\wid+11.5){$\vdots$};

    \foreach \j in {0,...,5} {
        \pgfmathsetmacro{\x}{((\j*4)+2)*\wid+0.5*\wid}%
        \pgfmathsetmacro{\y}{(-(\j*4)+23)*\wid+0.5*\wid}%
        \node[] at (\x,\y) {$1$};
        \draw[gray] (\x+0.2*\wid,\y) -- (12,\y);
        \draw[gray] (\x, \y-0.3*\wid) -- (\x, 0);
    }
    \foreach \j in {0,...,5} {
        \pgfmathsetmacro{\x}{((\j*4)+1)*\wid+0.5*\wid}%
        \pgfmathsetmacro{\y}{(-(\j*4)+23)*\wid+0.5*\wid}%
        \node[] at (\x,\y) {$e$}; 
    }

    \foreach \j in {0,...,5} {
        \pgfmathsetmacro{\x}{((\j*4)+0)*\wid+0.5*\wid}%
        \pgfmathsetmacro{\y}{(-(\j*4)+21)*\wid+0.5*\wid}%
        \node[] at (\x,\y) {$e$}; 
    }
    
    \foreach \j in {0,...,5} {
        \pgfmathsetmacro{\x}{((\j*4)+1)*\wid+0.5*\wid}%
        \pgfmathsetmacro{\y}{(-(\j*4)+22)*\wid+0.5*\wid}%
        \node[] at (\x,\y) {$1$};
        \draw[gray] (\x+0.2*\wid,\y) -- (12,\y);
        \draw[gray] (\x, \y-0.3*\wid) -- (\x, 0);
    }
    
    \foreach \j in {0,...,5} {
        \pgfmathsetmacro{\x}{((\j*4)+3)*\wid+0.5*\wid}%
        \pgfmathsetmacro{\y}{(-(\j*4)+21)*\wid+0.5*\wid}%
        \node[] at (\x,\y) {$1$};
        \draw[gray] (\x+0.2*\wid,\y) -- (12,\y);
        \draw[gray] (\x, \y-0.3*\wid) -- (\x, 0);
    }

    \foreach \j in {0,...,5} {
        \pgfmathsetmacro{\x}{((\j*4)+0)*\wid+0.5*\wid}%
        \pgfmathsetmacro{\y}{(-(\j*4)+20)*\wid+0.5*\wid}%
        \node[] at (\x,\y) {$1$};
        \draw[gray] (\x+0.2*\wid,\y) -- (12,\y);
        \draw[gray] (\x, \y-0.3*\wid) -- (\x, 0);
    }
    
    \draw[blue, thick] (14*\wid, 11*\wid) -- (9*\wid, 11*\wid);
    \draw[blue, thick] (14*\wid, 11*\wid) -- (14*\wid, 24*\wid);
    \draw[blue, thick] (9*\wid, 11*\wid) -- (9*\wid, 24*\wid);
    
    \draw[red, thick] (13*\wid, 9*\wid) -- (9.1*\wid, 9*\wid);
    \draw[red, thick] (13*\wid, 9*\wid) -- (13*\wid, 24*\wid);
    \draw[red, thick] (9.1*\wid, 9*\wid) -- (9.1*\wid, 24*\wid);

\end{tikzpicture}
    \caption{Computing $\mathcal{X}^w_\circ \cap \mathcal{X}_v$ for $w=[-5,5,0,10]$ and $v=[-1,1,6,4]$}
    \label{fig:bigex}
\end{figure}

\end{example}

\subsection{Equations for Kazhdan-Lusztig varieties in $Fl_0(V)$}
Let $v,w\in W_0$ such that $v\le w$. Let $\mathcal{X}^w_\circ$ be parametrized by $m_{Q(w)}$ as in Section~\ref{subsec:param}. We are ready to describe the equations defining $\mathcal{X}^{w}_{\circ,v}$ in these coordinates.

\begin{lemma}
\label{lem:wvwindow}
Let $v,w\in W_0$ such that $v\le w$.
For all $(i,j)\in \mathbb{Z}\times \mathbb{Z}$, $l_w(i,j)\ge l_v(i,j)$.
\end{lemma}

\begin{proof}
    Since $v\le w$, we have $v^{-1}\le w^{-1}$. We then have for any $i,j$,
    \[|\Z_{\le i}-v^{-1}\Z_{\le j}| = \dim(E_i/(E^{v^{-1}}_j\cap E_i))\le \dim (E_i/(E^{w^{-1}}_j\cap E_i))=|\Z_{\le i}-w^{-1}\Z_{\le j}|.\]
    Note that for any affine permutation $w$, the quantity $|\Z_{\le i}-w^{-1}\Z_{\le j}|$ is the number of 1's strictly east of column $j$ and weakly north of row $i$ in the permutation matrix of $w^{-1}$, and therefore
    is also the number of 1's strictly south of row $j$ and weakly west of column $i$ in the permutation matrix of $w$.
    Suppose to the contrary that there is some $(j,i)$ such that $l_w(j,i)<l_v(j,i)$.
    Then at the coordinate $p=(j,i-l_w(j,i))$, the number of 1's  strictly south and weakly west of $p$ is 0 in the permutation matrix of $w$ but positive in the permutation matrix of $v$. Contradiction.
\end{proof}

Let $M:=\phi(m_{Q(w)}(x_1,\cdots, x_{\ell(w)}))$. 
For each $(i,j)\in \ess(v)$, Let $EQ^{i,j}_{w,v}$ denote the set of all minors of size $n_v(i,j,l_w(i,j))+1$ in the submatrix $M_{(-\infty, i],[j-l_w(i,j)+1, j]}$. The subset of nonzero minors is a finite set, since the submatrix $M_{(-\infty, i],[j-l_w(i,j)+1, j]}$ has all zeroes in rows with sufficiently small indices.

\begin{lemma}
\label{lem:stable}
Suppose $(i,j)\in \ess(v)$ and $l\ge l_w(i,j)$. Then the ideal generated by the set $EQ$ of minors of size $n_v(i,j,l)+1$ in the submatrix $M_{(-\infty,i],[j-l+1,j]}$ is equal to the ideal generated by $EQ^{i,j}_{w,v}$, the set of minors of size $n_v(i,j,l_w(i,j))+1$ in the submatrix $M_{(-\infty,i],[j-l_w(i,j)+1,j]}$. Furthermore, $EQ^{i,j}_{w,v}\subset EQ$.

\end{lemma}
\begin{proof}
Suppose the statement is true for $l\ge l_w(i,j)$ and we will show that it is true for $l+1$. First note that for all $b\le j-l_w(i,j)$, $v(b)\le i$ and $w(b)\le i$. Therefore, $n_v(i,j,l+1)=n_v(i,j,l)+1$. It is easy to see that a minor of size $n_v(i,j,l)+2$ in $M_{(-\infty,i],[j-l,j]}$ is in the ideal generated by a minor of size $n_v(i,j,l)+1$ in $M_{(-\infty,i],[j-l+1,j]}$. We just need to show that each nonzero minor of size $n_v(i,j,l)+1$ in $M_{(-\infty,i],[j-l+1,j]}$ is also a minor of size $n_v(i,j,l)+2$ in $M_{(-\infty,i],[j-l,j]}$. Suppose $\det M_{A,B}$ is a nonzero minor of size $n_v(i,j,l)+1$ in $M_{(-\infty,i],[j-l+1,j]}$. If it were the case that $w(j-l)\in A$, then $M_{A,B}$ would contain a row of zeroes, causing the determinant to vanish. Therefore $w(j-1)\not\in A$. Then Let $A':=A\cup\{w(j-l)\}$ and $B':=B\cup\{j-l\}$. Then $\det M_{A',B'}=\det M_{A,B}$, since row $w(j-l)$ has a unique 1 in column $j-l$.
\end{proof}

\begin{prop}
\label{prop:setKL}
The set of equations \[EQ_{w,v}:=\bigcup_{(i,j)\in \ess(v)}EQ_{w,v}^{i,j}=\coprod_{\substack{(i,j)\in \ess(v)\\1\le j\le n}}EQ_{w,v}^{i,j}\] are
set-theoretic equations for $\mathcal{X}_v^w$.
\end{prop}
\begin{proof}
The  equality between unions follows from the periodicity of affine $\infty\times\infty$ matrices.
 By Theorem~\ref{thm:seteq} and Lemma~\ref{lem:wvwindow}, if $\phi(F)$ for $F\in \Im(m_{Q(w)})$ satisfies all equations in $EQ$, we have $[F]\in \mathcal{X}^w_v$. On the other hand, given $F\in \Im(m_{Q(w)})$ such that $[F]\in\mathcal{X}_v$, by Theorem~\ref{thm:seteq}, for each $i,j\in \ess(v)$ there exists some $l\ge l_v(i,j)$ such that all minors of size $n_v(i,j,l)+1$ in the submatrix $F_{(-\infty, i],[j-l+1,j]}$ vanish. Note that if $l'\ge l$, the vanishing of all minors of size $n_v(i,j,l)+1$ in the submatrix $F_{(-\infty, i],[j-l+1,j]}$ implies the vanishing of all minors of size $n_v(i,j,l')+1$ in the submatrix $F_{(-\infty, i],[j-l'+1,j]}$. By Lemma~\ref{lem:stable}, the minors determined by $l=l_w(i,j)$ give the weakest conditions, and therefore have to vanish.
\end{proof}
 
 \begin{example}
\label{ex:eqn}
Let $w=[-5,5,0,10]$ and $v=[-1,1,6,4]$.
$Q(w)$ is computed in Example~\ref{ex:chart}.
Then for $(3,4)\in \ess(v)$, $l_w(3,4)=5$, $n_v(3,4,5)=3$ (which is the number of ones in the blue window of the bottom matrix in Figure~\ref{fig:bigex}), so $EQ_{w,v}^{3,4}$ consists of the size 4 minors in the blue window of the top matrix in Figure~\ref{fig:bigex}. Explicitly, 
\begin{equation*} 
\begin{split}
 EQ_{w,v}^{3,4}  = \{ & x_{10}x_7x_3-x_{10}x_6x_5-x_9x_3+x_8x_5, \\
  & x_{11}x_7x_3-x_{11}x_6x_5+x_9x_8x_6-x_8^2x_7, \\
  & x_{11}x_3+x_{10}x_8x_6-x_8^2, \\
  & -x_{11}^2x_5-x_{11}x_{10}x_8x_7+x_{11}x_9x_8, \\
  & -x_{11}x_5-x_{10}x_8x_7+x_9x_8\}.
\end{split}
\end{equation*}

For $(5,3)\in \ess(v)$, $l_w(5,3)=4$, $n_v(5,3,5)=3$ (which is the number of ones in the red window of the bottom matrix in Figure~\ref{fig:bigex}), so $EQ_{w,v}^{5,3}$ consists of the size 4 minors in the red window of the top matrix in Figure~\ref{fig:bigex}.
Explicitly, 
\begin{equation*} 
\begin{split}
EQ_{w,v}^{5,3}=\{ & x_{10}x_7x_3-x_{10}x_6x_5-x_9x_3+x_8x_5, \\  
& x_{10}x_7x_2-x_{10}x_6x_4-x_9x_2+x_8x_4, x_{10}x_5x_2-x_{10}x_4x_3, x_5x_2-x_4x_3, \\ & x_{10}x_7x_1-x_{10}x_6-x_9x_1+x_8, x_{10}x_5x_1-x_{10}x_3, x_5x_1-x_3, \\
& x_{10}x_4x_1-x_{10}x_2, x_4x_1-x_2\}.
\end{split}
\end{equation*}

Finally, $EQ_{w,v}=EQ_{w,v}^{3,4}\sqcup EQ_{w,v}^{5,3}$.

 \end{example}

 \section{Gr\"obner basis for Kazhdan-Lusztig ideals}
  \label{sec:gb}
  
 In a polynomial ring $\mathbb{C}[x_1,\cdots, x_k]$, we fix a monomial ordering to be lexicographic with $x_1<x_2<\cdots <x_k$.
 Let $v,w\in W_0$ with $v\le w$. The \textbf{lead term} of a polynomial $g$ is the largest monomial summand of $g$ with respect to this monomial order.
 Let $I_{w,v}:=I_{w,v}^{Q(w)}\subset \mathbb{C}[x_1,\cdots, x_{\ell(w)}]$ be the ideal for the Kazhdan-Lusztig variety $\mathcal{X}_v^w$ in the Bott-Samelson coordinate determined by $Q(w)$. The goal of this section is to prove the following theorem:
  \begin{thm}
  The polynomials in $EQ_{w,v}$ form a Gr\"obner basis for the Kazhdan-Lusztig ideal $I_{w,v}$.
  \end{thm}
 Let $Q=(s_{\alpha_1},\cdots, s_{\alpha_{\ell(w)}})$ be a reduced word for $w$.  
 In $\mathbb{C}[x_1,\cdots, x_{\ell(w)}]$, define
 \[SR(Q,v):=\bigcap_{\substack{Q'\subset Q \text{ reduced subword}\\ \prod Q' =v}} I_{Q'},\]
 where $I_{Q'}:=\tup{x_i:{s_{\alpha_i}\in Q'}}$.
 The following theorem is due to Knutson, see \cite[Theorem 4]{knutson2008schubert}. The statement is type-independent.
 \begin{thm}
 \label{thm:knutson}
 If $Q$ is a reduced word for $w$ and $I_{w,v}^Q$ is the ideal for the Kazhdan-Lusztig variety $\mathcal{X}^w_v$ in the Bott-Samelson coordinates determined by $Q$, then 
 \[\init I^Q_{w,v}=SR(Q,v).\]
 \end{thm}
\begin{rmk}
\label{rmk:KLinit}
    Although the general setting in \cite{knutson2008schubert} assumes the flag variety to be finite dimensional, we note that this assumption is not necessary for the statements about Kazhdan--Lusztig varieties. In particular, \cite[Theorem 4]{knutson2008schubert} relies on \cite[Theorem 2']{knutson2008schubert}, which can be proved by direct analogues of \cite[Proposition 6 and 7]{knutson2007automatically} by degenerating (in Knutson's notation where $w\le v$) $X_w\cap (X^{v}_\circ \cup X^{vr_{\alpha}}_\circ)$ inside $X^{v}_\circ \cup X^{vr_{\alpha}}_\circ$. The proofs for Kazhdan-Lusztig varieties are simpler and we give a sketch in Appendix~\ref{app}. These arguments do not assume finite-dimensionality of the ambient flag varieties. Relatedly, in \cite[Theorem 7]{knutson2009frobenius}, Knutson also showed that this degeneration is compatible with the Frobenius splitting on $X^v_\circ$ that compatibly splits the Kazhdan-Lusztig varieties in a  type-independent manner. 
\end{rmk}

 \begin{rmk} 
 \label{rmk:equivariant}
 A formula for the $K$-polynomials of Stanley-Reisner rings of  subword complexes is given in
\cite[Theorem 4.1]{knutson2004subword}. 
For $Q=(s_{\alpha_1},\cdots, s_{\alpha_{\ell(w)}})$ a reduced word for $w$, set $\beta_i:=s_{\alpha_1}\cdots s_{\alpha_{i-1}}\widetilde{\alpha}_i \in \Z[y_1^\pm,\cdots, y_n^\pm]$, where $\widetilde{\alpha_i}:=y_{\alpha_i}/y_{\alpha_i+1}$ if $1\le \alpha_i \le n-1$ and $\widetilde{0}=y_n/y_1$. 
Setting the (exponential) weight of the variable $x_i$ to be $\beta_i^{-1}$ and applying \cite[Theorem 4.1]{knutson2004subword}, we get
\[\mathcal{K}(I^Q_{w,v};\mathbf{y}) = \mathcal{K}(SR(Q,v);\mathbf{y})=\sum_{\substack{Q'\subset Q \text{ subword}\\ \Dem Q' =v}} (-1)^{\ell(Q')-\ell(v)} \prod_{s_{\alpha_i}\in Q'} (1-\beta_i^{-1}),\]
we recover \cite[Theorem 2.6]{graham2015excited} formula in the affine type $A$. (We remark that Willem \cite[Theorem 4.7]{willems2004equivariant}) has a similar formula in the Kac-Moody setting with a different basis.) If we instead weigh the variable $x_i$ with $\beta_i$ and consider the lowest degree term of the twisted $K$-polynomial $\mathcal{K}(I^Q_{w,v};\mathbf{1-y})$ (cf.\cite[Section 1.2]{knutson2005grobner}), we are able to obtain Andersen–Jantzen–Soergel\cite{andersen1994representations}/Billey\cite{billey}'s formula in $T$-equivariant cohomology. See \cite{knutson2008schubert} for a similar discussion.
 \end{rmk}

 Let \[J_{w,v}:=\tup{\LT(g):g\in EQ_{w,v}},\]
 where $\LT(g)$ denotes the lead term of $g$ in the fixed monomial order.
 By Proposition~\ref{prop:setKL}, $EQ_{w,v}\subset I_{w,v}$, so $J_{w,v}\subset \init I_{w,v}$. We will show the following key technical statement:
 \begin{prop}
  \label{prop:SRinJ}
  $SR(Q(w),v)\subset J_{w,v}$.
 \end{prop}
  This, together with Theorem~\ref{thm:knutson}, imply that $EQ_{w,v}$ form a Gr\"obner basis for $I_{w,v}$. The structure of our technical proofs follows that of \cite{woo2012grobner}, with some details modified to suit the affine type $A$ setting. Before proceeding with the proof, we first demonstrate the Gr\"obner basis statement for Example~\ref{ex:eqn}.
  
  \begin{example}
  Let $w=[-5,5,0,10]$ and $v=[-1,1,6,4]$ as in Example~\ref{ex:eqn}. Recall that
  \[Q(w)=s_1s_0s_3s_1s_0s_2s_3s_1s_2s_3s_0.\]
  The reduced words of $v$ consists of $s_1s_0s_3s_0$, $s_1s_3s_0s_3$, and $s_3s_1s_0s_3$. The possible subwords of $Q(w)$ that are reduced words for $v$ are illustrated in the following table:
  \begin{center}
  \scalebox{0.8}{
  \begin{tabular}{|c|c|c|c|c|c|c|c|c|c|c|}
  \hline
    $s_1$ & $s_0$ & $s_3$ & $s_1$ & $s_0$ & $s_2$ & $s_3$ & $s_1$ & $s_2$ & $s_3$ & $s_0$ \\ \hline
    $-$ & $-$ & $-$ &  & $-$ &  &  &  &  &  &  \\ \hline
    $-$ & $-$ & $-$ &  &     &  &  &  &  &  & $-$ \\ \hline
    $-$ & $-$ &     &  &     &  & $-$ &  &  &  & $-$ \\ \hline
    $-$ & $-$ &     &  &     &  &  &  &  & $-$ & $-$ \\ \hline
    $-$ &     & $-$ &  & $-$ &  & $-$ &  &  &  &  \\ \hline
    $-$ &     & $-$ &  & $-$ &  &     &  &  & $-$ &  \\ \hline
    $-$ &     &     &  & $-$ &  & $-$ &  &  &  & $-$ \\ \hline
    $-$ &     &     &  & $-$ &  &     &  &  & $-$ & $-$ \\ \hline
        &     & $-$ & $-$ & $-$ &  & $-$ &  &  &  &  \\ \hline
        &     & $-$ & $-$ & $-$ &  &  &  &  & $-$ &  \\ \hline
        &     &     & $-$ & $-$ &  & $-$ &  &  &  & $-$ \\ \hline
        &     &     & $-$ & $-$ &  &     &  &  & $-$ & $-$ \\ \hline
  \end{tabular} }
  \end{center}
  On the other hand, the ideal generated by the lead terms of the polynomials in $EQ_{w,v}$ is \[
  \begin{aligned}
    J_{w,v} = &\langle x_{10}x_7x_3, x_{10}x_7x_2,x_5x_2, x_{10}x_7x_1,x_5x_1,x_4x_1,x_{11}x_3,x_{11}x_5\rangle \\
   = & \langle x_1,x_2,x_3,x_5 \rangle \cap \langle x_1,x_2,x_3,x_{11} \rangle \cap \langle x_1,x_2,x_7,x_{11} \rangle  \\
  & \cap \langle x_1,x_2,x_{10},x_{11} \rangle  
   \cap \langle x_1,x_3,x_5,x_7 \rangle 
   \cap \langle x_1,x_3,x_5,x_{10} \rangle \\ 
  & \cap \langle x_1,x_5,x_7,x_{11} \rangle 
   \cap \langle x_1,x_5,x_{10},x_{11} \rangle  
   \cap \langle x_3,x_4,x_5,x_7 \rangle  \\
  & \cap \langle x_3,x_4,x_5,x_{10} \rangle 
   \cap \langle x_4,x_5,x_7,x_{11} \rangle  
   \cap \langle x_4,x_5,x_{10},x_{11} \rangle \\
  & =SR(Q(w),v)
  \end{aligned} 
  \]
  \end{example}

We recall the basics of subword complexes, defined and studied in \cite{knutson2004subword}. For any reduced word $Q=(s_{\alpha_1},\cdots, s_{\alpha_{\ell(w)}})$ for $w$, and $v$ a permutation with $v\le w$,  define the \textbf{subword complex} $\Delta(Q,v)$ as follows. Let $V:=\{1,2,\cdots, \ell(w)\}$ be the vertex set of $\Delta(Q,v)$, and $F\subset V$ be a facet if and only if the subword of $Q$ formed by simple reflections in positions $V\setminus F$, which we denote as $Q_{V\setminus F}$, is a reduced word for $v$. More generally, $F\in\Delta(Q,v)$ is a face if and only if the Demazure product $\Dem(Q_{V\setminus F})\ge v$. Recall that the Demazure product  of a (not necessarily reduced) word can be defined inductively as $\Dem(Q,s_i)=\Dem(Q)$ if $\ell(\Dem(Q)s_i)<\ell(\Dem(Q))$, and $\Dem(Q)s_i$ otherwise.

By the Stanley-Reisner correspondence, we have \[SR(Q,v)=\tup{\prod_{a\in  P} x_{a}:P\not\in \Delta(Q,v)}.\] In other words, $SR(Q,v)$ is generated by monomials given by the nonfaces of $\Delta(Q,v)$. 

We collect some facts about vertex decomposition for abstract simplicial complexes and in particular, subword complexes. 

For any simplicial complex $\Delta$ and $q$ a vertex of $\Delta$,
the \textbf{deletion} of $q$ is the set of faces that do not contain $q$:
\[\del_q(\Delta)=\{F\in \Delta: q\not\in F\}.\]
The \textbf{link} of $q$ is 
\[\link_q(\Delta)=\{F\in \del_q(\Delta): F\cup\{q\}\in\Delta \}\]

The following two lemmas give some useful inductive properties of deletion and link for simplicial complexes (Lemma~\ref{lem:simplicial_induction}) and subword complexes (Lemma~\ref{lem:subword_induction}).
\begin{lemma}
\label{lem:simplicial_induction}
\cite[Lemma 6.3]{woo2012grobner} Let $S$ be a subset of the vertices of $\Delta$, and $q$ a fixed vertex of $\Delta$. Then $S\not\in \Delta$ if and only if either (a) or (b) holds:
\begin{enumerate}[(a)]
    \item $q\in S$, and $S\setminus\{q\}\not\in\link_q(\Delta)$
    \item $q\not\in S$, and $S\not\in \del_q(\Delta)$.
\end{enumerate}
\end{lemma}

\begin{lemma}
\label{lem:subword_induction}
\cite[Theorem 6.4]{woo2012grobner}
Suppose $Q=(Q',s_\sigma)$ and $q=\ell(Q)$. Then 
\begin{enumerate}[(a)]
    \item If $\ell(vs_\sigma)>v$, then $\link_q(\Delta(Q,v))=\del_q(\Delta(Q,v))=\Delta(Q', v)$. 
    \item If $\ell(vs_\sigma)<v$, then $\link_q(\Delta(Q,v))=\Delta(Q',v)$ and $\del_q(\Delta(Q,v))=\Delta(Q',vs_\sigma)$.
\end{enumerate}
\end{lemma}

The following Lemma parallels \cite[Proposition 6.15]{woo2012grobner}. While it shares the same structure, affine type $A$ presents additional technical challenges that have to be addressed specifically in the proof. We also supply worked out examples after the proof for clarity.
\begin{lemma}
\label{lem:main}
Let $\sigma:=\sigma(w)$, namely, $Q(w)=(Q(ws_\sigma),s_\sigma)$. 
\begin{enumerate}[(a)]
    \item \label{thm:casea} Suppose $\ell(vs_\sigma)>\ell(v)$ and $g\in EQ_{ws_\sigma,v}$, then there exists $g'\in EQ_{w,v}$ such that $\LT(g')$ divides $\LT(g)$.
    \item \label{thm:caseb} Suppose $\ell(vs_\sigma)<\ell(v)$ and $g\in EQ_{ws_\sigma,v}$, then there exists $g'\in EQ_{w,v}$ such that $\LT(g')$ divides $x_{\ell(w)}\LT(g)$.
    \item \label{thm:casec} Suppose $\ell(vs_\sigma)<\ell(v)$ and $g\in EQ_{ws_\sigma,vs_\sigma}$, then there exists $g'\in EQ_{w,v}$ such that $\LT(g')$ divides $\LT(g)$.
\end{enumerate}
\end{lemma}
\begin{proof}
Let $M:=\phi(m_{Q(ws_\sigma)}(\mathbf{x}))$ and $M'=\phi(m_{Q(w)}(\mathbf{x}))$. 

\noindent \textbf{Case (a).} 
Suppose $g\in EQ^{i,j}_{ws_\sigma,v}$, a minor of size $n_v(i,j,l_{ws_\sigma}(i,j))+1$ that is given by $(i,j)\in \ess(v)$.
Notice that $l_w(i,j)=l_{ws_{\sigma}}(i,j)$ or $l_w(i,j)=l_{ws_{\sigma}}(i,j)+1$. In the latter case, by Lemma~\ref{lem:stable}, we may consider $g$ instead as a minor of size $n_v(i,j,l_{ws_\sigma}(i,j))+2=n_v(i,j,l_{w}(i,j))+1$ in $M_{(-\infty,i],[j-l_w(i,j)+1,j]}$.
 Suppose 
\[g=\det M_{A,B}\]
where $A\subset \Z$ is a set of row indices, $B\subset \Z$ is a set of column indices, and $|A|=|B|=n_v(i,j,l_{w}(i,j))+1$. We will construct a set $B'$ such that $g':=\det M'_{A,B'}\in EQ_{w,v}^{i,j}$, and $\LT(g')$ divides $\LT(g)$. Recall that $M'$ can be obtained from $M$ by replacing the $0$ with $x_{\ell(w)}$ at position $(ws_{\sigma}(b), b+1)$, and swapping columns $b$ and $b+1$ for each $b\in [\sigma]_n$. We call this operation $\bsswap(b)$ for each $b\in [\sigma]_n$. 

To initialize, set $B_1:=B$ and $M_1:=M$. Also set $B_{\operatorname{used}}=\emptyset$. We give an algorithm that iteratively performs $\bsswap(b)$ in a certain order on $M_1$ that eventually gives $M'$ and also modifies $B_1$ that eventually gives $B'$. (In other words, we treat $M_1$ and $B_1$ as ``mutable objects''.)
\begin{itemize}
    \item[Step 1.] For each $b\in B$, if $b\in [\sigma]_n$ but $b+1\not\in B$: 
    \begin{itemize}
        \item Replace $b$ with $b+1$ in $B_1$;
        \item Add $b$ to $B_{\operatorname{used}}$;
        \item Perform $\bsswap(b)$ on $M_1$.
        
    \end{itemize}
    \item[Step 2.] For each $b\in B$ in increasing order, if $b\in[\sigma+1]_n$ but $b-1\not\in B$:
    \begin{itemize}
    \item If $ws_\sigma(b-1)\not\in A$: 
    \begin{itemize}
        \item Replace $b$ with $b-1$ in $B_1$;
        \item Add $b-1$ to $B_{\operatorname{used}}$;
        \item Perform $\bsswap(b-1)$ on $M_1$.
        
    \end{itemize}
    \item If $ws_\sigma(b-1)\in A$:
    \begin{itemize}
        \item Let $c$ be the column index of the entry in row $ws_\sigma(b-1)$ of $M_1$ that contributes to $\LT(\det (M_1)_{A,B_1})$.
        Replace $c$ with $b-1$ in $B_1$;
        \item Add $b-1$ to $B_{\operatorname{used}}$;
        \item Perform $\bsswap(b-1)$ on $M_1$.
    \end{itemize}
    \end{itemize}
    \item[Step 3.] For each $b\in[\sigma]_n\setminus B_{\operatorname{used}}$, perform $\bsswap(b)$ on $M_1$.
\end{itemize}

After the final step, $M_1=M'$. We set $B':=B$. We now argue that the algorithm is valid, and after each operation of $\bsswap$ in the process described above, $B_1\subset [j-l_w(i,j)+1,j]$ and  $\LT(\det ((M_1)_{A,B_1}))$ divides $\LT(M_{A,B})$.

First, we must argue that if $b\in B$ and $b\in[\sigma]_n$, then $b+1 \le j$. Suppose otherwise, then it must be the case that $b=j$. Since $\ell(vs_\sigma)>\ell(v)$, $v(b+1)>v(b)$, and therefore $\ess(v)$ does not contain any boxes in columns $b\in[\sigma]_n$, so $j\not\in [\sigma]_n$. But this means $b\not\in [\sigma]_n$; contradiction. Therefore $B_1 \subset [j-l_w(i,j)+1,j]$ during Step 1. Furthermore, before performing $\bsswap(b)$ on $M_1$ in Step 1, the column $b$ of $M_1$ contains a single 1 in row $ws_\sigma(b)=w(b+1)$, and after the operation the column $b+1$ of $M_1$ contains a single 1 in the same row. (Note also that for this reason if $\det(M_1)_{A,B_1}\neq 0$ we must have $ws_\sigma(b)=w(b+1)\in A$.) Therefore, the submatrix  $(M_1)_{A,B_1}$ does not change as a square matrix during Step 1. 

Suppose we are in Step 2. 
First we argue that if $b\in B$ and $b\in [\sigma+1]_n$, then $b-1\ge j-l_w(i,j)+1$. Suppose otherwise, then it must be the case that $b=j-l_w(i,j)+1$, so $j-l_w(i,j)\in[\sigma]_n$. By definition of $l_w(i,j)$, we must have $w(j-l_w(i,j))<w(j-l_w(i,j)+1)$, so $w(b-1)<w(b)$. This contradicts that $ws_\sigma < w$, since $b-1\in [\sigma]_n$. 
Therefore $B_1 \subset [j-l_w(i,j)+1,j]$ during Step 2.

Now consider an iteration of the loop in the case $ws_\sigma(b-1)\not\in A$ in Step 2.
In this case, the column $(M_1)_{A,b}$ before performing $\bsswap(b-1)$ is the same as column $(M_1)_{A,b-1}$ after the operation. Therefore, the submatrix  $(M_1)_{A,B_1}$ does not change as a square matrix in this case of Step 2.

Now consider an iteration of the loop in the case $ws_\sigma(b-1)\in A$ in Step 2. First note that $(M_1)_{w(b-1),c}$ is not 1, because the 1 in row $ws_\sigma(b-1)$ is in column $b-1$, and by assumption $b-1\not\in B$ and was not added to $B_1$ in Step 1. Furthermore, since $b-1\in [\sigma]_n$, all entries in row $ws_\sigma(b-1)$ of $M_1$ to the right of position $(ws_\sigma(b-1),b-1)$ are 0, so we must have $c<b-1$. Therefore, since we iterate through $b\in B$ in increasing order, $c$ cannot be added back into $B_1$ in the following iterations of the loop.

Before performing $\bsswap(b-1)$, 
\[(M_1)_{A,c}=\left(\begin{smallmatrix}
   * \\
   \vdots \\
   * \\
   (M_1)_{w(b-1),c} \\
   * \\
   \vdots\\
   * 
\end{smallmatrix}\right), (M_1)_{A,b}=\left(\begin{smallmatrix}
   0 \\
   \vdots \\
   0 \\
   0_{w(b-1),b} \\
   d_1 \\
   \vdots\\
   d_k
\end{smallmatrix} \right),\]
where subscript on the 0 specifies its position in $M_1$, and $d_1,\cdots, d_k$ are corresponding matrix entries.  After performing $\bsswap(b-1)$, 
\[(M_1)_{A,b-1}=\left(\begin{smallmatrix}
   0 \\
   \vdots \\
   0 \\
   x_{\ell(w)} \\
   d_1 \\
   \vdots\\
   d_k
\end{smallmatrix} \right),\hskip 1.5em   (M_1)_{A,b}=\left(\begin{smallmatrix}
   0 \\
   \vdots \\
   0 \\
   1 \\
   0 \\
   \vdots\\
   0 
\end{smallmatrix}\right),\]
Let $g_\text{before}$ and $g_\text{after}$ denote $\det (M_1)_{A,B_1}$ before and after the iteration of this loop. Since our monomial order is a lex order, $\LT(g_\text{before})=\pm (M_1)_{w(b-1),c}\LT(g_\text{after})$.

In Step 3, the operations $\bsswap(b)$ do not change $\pm\det(M_1)_{A,B_1}$, because if $b_1\in B_1$, $b_1\not\in [\sigma]_n$ and $b_1\not\in [\sigma+1]_n$, $\bsswap(b)$ does not affect column $b_1$. Otherwise, for each $b\in [\sigma_n]\setminus B_\text{used}$, it must be the case that $b\in B_1$ if and only if $b+1 \in B_1$. It is easy to see that in this case $\bsswap(b)$ does not change $\pm\det(M_1)_{A,B_1}$.

Therefore, the algorithm constructs $B'$ such that $\LT(g')$ divides $\LT(g)$. Example \ref{ex:Case(a)} shows an instance of this case.
\vskip 1em

 \noindent\textbf{Case (b).} The argument for this case is mostly the same for Case (a), except for the part in Step 1 of the algorithm where the assumption $\ell(vs_\sigma)>\ell(v)$ is used. We modify Step 1 in the algorithm of Case (a) as follows:
 
\begin{itemize}
    \item[Step 1.] For each $b\in B$, if $b\in [\sigma]_n$ but $b+1\not\in B$: 
    \begin{itemize}
        \item Replace $b$ with $b+1$ in $B_1$ if $b+1 \le j$, otherwise do nothing.
        \item Add $b$ to $B_{\operatorname{used}}$;
        \item Perform $\bsswap(b)$ on $M_1$.
    \end{itemize}
\end{itemize}
 
We now address the case when $b=j$ in Step 1. This is possible, because $\ell(vs_\sigma)<\ell(v)$ implies $v$ has at least one essential box in each column in $[\sigma]_n$. Note that for the minor under consideration to be nonzero we must have $ws_\sigma(j)=w(j+1)\in A$.
Before $\bsswap(j)$ is performed, 
\[(M_1)_{A,j}=\left(\begin{smallmatrix}
   0 \\
   \vdots \\
   0 \\
   1_{ws_\sigma(j),j} \\
   0 \\
   \vdots\\
   0 
\end{smallmatrix}\right)\]
where the subscript on the 1 specifies its matrix coordinate, and after 
$\bsswap(j)$ is performed, 
\[(M_1)_{A,j}=\left(\begin{smallmatrix}
   0 \\
   \vdots \\
   0 \\
   x_{\ell(w)} \\
   * \\
   \vdots\\
   * 
\end{smallmatrix}\right),\]
where $x_{\ell(w)}$ is in position $(ws_\sigma(j),j)=(w(j+1),j)$ of $M_1$. 

Let $g_\text{before}$ and $g_\text{after}$ denote $\det (M_1)_{A,B_1}$ before and after $\bsswap(j)$ is performed. Since $x_{\ell(w)}$ is the lex-largest variable, we must have $x_{\ell(w)}\LT(g_\text{before})=\LT(g_\text{after})$. (Note that $b$ can equal to $j$ exactly once, so we cannot have higher powers of $x_{\ell(w)}$.)
The rest of the argument proceeds the same as in Case (a), and the algorithm constructs $B'$ such that for $g=\det M_{A,B}$ and $g'=\det M'(A,B')$,
$\LT(g')$ divides $x_{\ell(w)}\LT(g)$. Example \ref{ex:Case(b)} shows an instance of this case.
\vskip 1em

\noindent\textbf{Case (c).} If $(i,j)\in\ess(vs_\sigma)$, then $(i,j)\in \ess(v)$, $(i,j+1)\in\ess(v)$, or $(i,j-1)\in \ess(v)$. (Note that it's impossible to have $j \in [\sigma]_n$.) When $(i,j)\in \ess(v)$, we note that $n_{vs_\sigma}(i,j,l_w(i,j))=n_{v}(i,j,l_w(i,j))$. This follows from $l_w(i,j)\ge l_v(i,j)$ by Lemma~\ref{lem:wvwindow}, and therefore $s_\sigma$ cannot switch out a 1 in the window considered. The proof then proceeds in the same way as in Case (a). 
The case where $(i,j+1)\in \ess(v)$ is similar.
Finally, when $(i,j-1)\in \ess(v)$, it is necessarily true that $j\in [\sigma+1]_n$. Let $g\in EQ^{i,j}_{vs_\sigma, ws_\sigma}$.

\vskip 0.5em

\textbf{Subcase $j-l_{ws_\sigma}(i,j)\not\in [\sigma]_n$.} 
In this case, $l_{ws_\sigma}(i,j) = l_w(i,j)$, so $l_w(i,j-1)=l_{ws_\sigma}(i,j)-1$, and $n_{vs_\sigma}(i,j,l_{ws_\sigma}(i,j))=n_v(i,j-1,l_w(i,j-1))+1$. Therefore, $EQ^{i,j}_{ws_\sigma, vs\sigma}$ consists of minors of size $N:=n_{vs_\sigma}(i,j,l_{ws_\sigma}(i,j))+1$ in 
\[M_{(-\infty, i],[j-l_{ws_\sigma}(i,j)+1,j]}=M_{(-\infty, i],[j-l_{w}(i,j)+1,j]}\]
and $EQ^{i,j-1}_{w,v}$ consists of minors of size $N-1$ in \[M'_{(-\infty, i],[j-1-l_w(i,j-1)+1,j-1]}=M'_{(-\infty, i],[j-l_{ws_\sigma}(i,j)+1,j-1]}=M'_{(-\infty, i],[j-l_{w}(i,j)+1,j-1]}.\]
Suppose $g=\det M_{A,B}$, where $|A|=|B|=N$.

Now we repeat the construction in Case (b) and obtain $B'$ such that $g'=\det M'_{A,B'}$ is a minor of size $N$ in $M'_{(-\infty, i],[j-l_{w}(i,j)+1,j]}$. 
\begin{itemize}
    \item If $j\not\in B'$, we have $\LT(g')$ divides $\LT(g)$, because the construction in Case (b) in this case is identical to Case (a). We may take any $N-1$ out of the $N$ positions of $M'_{A,B'}$ that contribute to $\LT(g')$, and let $A'', B''$ be the corresponding row  and column indices determined by these $N-1$ positions. Then, $\LT(\det M'_{A'',B''})$ divides $\LT(g)$, and $\det M'_{A'',B''} \in EQ_{w,v}^{i,j-1}$. 
    
    \item If $j\in B'$, $\LT(g')$ contains the variable $x_{\ell(w)}$ and $\LT(g')$ divides $x_{\ell(w)}\LT(g)$. We take the $N-1$ positions of $M'_{A,B'}$ that contributes to $\LT(g')$ but not equal to $x_{\ell(w)}$, and let $A''$, $B''$ be the corresponding row and column indices. In other words, $A''=A\setminus\{w(j+1)\}$ and $B''=B'\setminus \{j\}$. Then, $\LT(\det M'_{A'',B''})$ divides $\LT(g)$, and $\det M'_{A'',B''} \in EQ_{w,v}^{i,j-1}$. 
\end{itemize}

\textbf{Subcase $j-l_{ws_\sigma}(i,j) \in [\sigma]_n$.} In this case $l_w(i,j)=l_{ws_{\sigma}}(i,j)+1$, so $l_w(i,j-1)=l_{ws_{\sigma}}(i,j)$. Since the number of 1's in the permutation matrix of $vs_\sigma$ strictly below row $i$ and weakly to the left of column $j$ are all strictly to the right of column $j-l_v(i,j)$ and $l_v(i,j)\le l_{ws_\sigma}(i,j)$ by Lemma~\ref{lem:wvwindow}, we see that \[N:=n_{vs_\sigma}(i,j,l_{ws_\sigma}(i,j))=n_v(i,j-1,l_w(i,j-1)).\]
Suppose $g=\det M_{A,B}$ as in the previous Subcase. Let $A_1:=A$. We then construct $A'$ and $B'$ using the algorithm described in Case (a) but modify Step 1 as follows:
\begin{itemize}
    \item[Step 1.]  For each $b\in B$, $b< j-1$ in increasing order, if $b\in [\sigma]_n$ but $b+1\not\in B$: 
    \begin{itemize}
        \item Replace $b$ in $B_1$ with $b+1$; 
        \item Add $b$ to $B_{\operatorname{used}}$;
        \item Perform $\bsswap(b)$ on $M_1$.
    \end{itemize}
    Special step:
    \begin{itemize}
        \item If $j-1\in B$ but $j\not \in B$, replace $j-1$ in $B_1$ with $(j-1)-l_w(i,j-1)+2=j-l_{ws_\sigma}(i,j)+1$ (note that $j-l_{ws_\sigma}(i,j)+1$ might already be in $B_1$; we keep both copies at this point and show one must be removed later), and replace $ws_\sigma(j)$ in $A_1$ with $ws_\sigma(j-l_{ws_\sigma}(i,j))=w((j-1)-l_w(i,j-1)+2)$;
    
        If $j-1\in B$ and $j\in B$, replace $j$ in $B_1$ with $(j-1)-l_w(i,j-1)+2=j-l_{ws_\sigma}(i,j)+1$, and replace $ws_\sigma(j)$ in $A'$ with $ws_\sigma(j-l_{ws_\sigma}(i,j))=w((j-1)-l_w(i,j-1)+2)$;
        
        \item Add $j-1$ to $B_{\operatorname{used}}$;
        \item Perform $\bsswap(j-1)$ on $M_1$.
    \end{itemize}

\end{itemize}
In Step 2, if $b_0:=j-l_{ws_\sigma}(i,j)+1 \in B$, we must be in the case $ws_\sigma(b_0-1)\not\in A$ because $M_{ws_\sigma(b_0-1),b_0-1}=1$ and column $b_0-1$ is outside of the submatrix of $M$ within which we consider minors. Therefore, (one copy of) $b_0$ in $B_1$ is replaced with $b_0-1$. 

Let $g_\text{before}$ and $g_\text{after}$ denote $\det (M_1)_{A_1,B_1}$ before the special step of Step 1 and after the first loop of Step 2. We see that $g_\text{before}$ and $g_\text{after}$ may only differ by sign, since these operations are equivalent to first removing a column containing a single 1 and the row containing this 1, adding a fresh row and column containing a (same) new 1, and finally also possibly adding $x_\ell(w)$ times the column containing the new 1 to another column. The rest of the argument is the same as in Case (a). At the end of the algorithm we set $A':=A_1$ and $B'=B_1$. Example \ref{ex:Case(c)2} shows an instance of this case.
\end{proof}

\begin{example}
  \label{ex:Case(a)} (Example computation for Case (\ref{thm:casea})) Let $w=[9,-7,4]$ and $v=[0,5,1]$.   We compute
  \[
    Q(w)=s_2s_1s_0s_2s_1s_2s_0s_1s_2s_0s_1.
  \]
  Then $(1,2)\in \ess(v)$. We compute $l_{ws_\sigma}(1,2)=7$ and $l_w(1,2)=8$ (see Figure \ref{fig:Ex(a)}). We also compute $n_v(1,2,7)=5$ and $n_v(1,2,8)=6$. Consider $g=\det M_{A,B}$, where $A=\{ -10,-7,-4,-2,-1,1 \},\text{ and }B=\{ -4,-3,-1,0,1,2 \}$. As we remarked, we can think of $g$ instead as a $7\times 7$ minor in $M_{(-\infty,1],[-5,2]}$ by including row $-13$ and column $-5$, so our choices are
  \[
    A=\{ -13,-10,-7,-4,-2,-1,1 \},\text{ and }B=\{ -5,-4,-3,-1,0,1,2 \}.
  \]
  We initialize $B_1:=B$ and follow the algorithm. Note that we have $\sigma=1$.
  \begin{itemize}
  \item[Step 1.] There are no columns in $B$ relevant to Step 1.
  \item[Step 2.] Since $-1\in [\sigma+1]_n$, but $-2\not\in B$, we have to consider $b=-1$ in Step 2. Since $ws_\sigma(-2)=-10\in A$, we are in the second subcase. The leading term of $\det (M_1)_{A,B}$ is $x_{10}x_9x_6x_4^2$, and the contribution from row $-10$ is the $x_{10}$ in column $c=-3$. We replace $-3$  with $-2$ in $B_1$ and perform $\bsswap(-2)$.
  \item[Step 3.] We perform $\bsswap(-5)$ and $\bsswap(1)$ and we are done, as the other $\bsswap$-s do not affect the submatrix.
  \end{itemize}
  After the algorithm we have
  \[
    A^\prime=\{ -13,-10,-7,-4,-2,-1,1 \},\text{ and }B^\prime=\{ -5,-4,-2,-1,0,1,2 \}.
  \]
  We have $\LT(\det(M^\prime)_{A^\prime,B^\prime})=-x_9x_6x_4^2$ which indeed divides $\LT(\det (M)_{A,B})=x_{10}x_9x_6x_4^2$.
  \begin{figure}[h!]
    \centering
    \begin{tikzpicture}[scale=0.8,transform shape]
      \pgfmathsetmacro{\wid}{0.5}
      \pgfmathsetmacro{\al}{0.4}
      \draw[LightSkyBlue1, step=\wid] (0,0) grid (4,10);
      
      \foreach  \j  in {1,...,8} {
        \pgfmathtruncatemacro{\x}{\j-6}%
        \node[font=\tiny] at (\j*\wid-0.5*\wid,20.2*\wid) {$\x$};
      };

      \foreach  \j  in {1,...,20} {
        \pgfmathtruncatemacro{\x}{7-\j}%
        \node[font=\tiny] at (-0.3,\j*\wid-0.5*\wid) {$\x$};
      };

      \draw[blue!20, line width=2.5mm, opacity=\al] (0.5*\wid, 5*\wid) -- (0.5*\wid, 20*\wid);
      \draw[blue!20, line width=2.5mm, opacity=\al] (1.5*\wid, 5*\wid) -- (1.5*\wid, 20*\wid);
      \draw[blue!20, line width=2.5mm, opacity=\al] (2.5*\wid, 5*\wid) -- (2.5*\wid, 20*\wid);
      \draw[blue!20, line width=2.5mm, opacity=\al] (4.5*\wid, 5*\wid) -- (4.5*\wid, 20*\wid);
      \draw[blue!20, line width=2.5mm, opacity=\al] (5.5*\wid, 5*\wid) -- (5.5*\wid, 20*\wid);
      \draw[blue!20, line width=2.5mm, opacity=\al] (6.5*\wid, 5*\wid) -- (6.5*\wid, 20*\wid);
      \draw[blue!20, line width=2.5mm, opacity=\al] (7.5*\wid, 5*\wid) -- (7.5*\wid, 20*\wid);
      
      \draw[red!20, line width=2.5mm, opacity=\al] (1*\wid, 5.5*\wid) -- (8*\wid, 5.5*\wid);
      \draw[red!20, line width=2.5mm, opacity=\al] (1*\wid, 7.5*\wid) -- (8*\wid, 7.5*\wid);
      \draw[red!20, line width=2.5mm, opacity=\al] (1*\wid, 8.5*\wid) -- (8*\wid, 8.5*\wid);
      \draw[red!20, line width=2.5mm, opacity=\al] (1*\wid, 10.5*\wid) -- (8*\wid, 10.5*\wid);
      \draw[red!20, line width=2.5mm, opacity=\al] (1*\wid, 13.5*\wid) -- (8*\wid, 13.5*\wid);
      \draw[red!20, line width=2.5mm, opacity=\al] (1*\wid, 16.5*\wid) -- (8*\wid, 16.5*\wid);
      \draw[red!20, line width=2.5mm, opacity=\al] (0*\wid, 19.5*\wid) -- (8*\wid, 19.5*\wid);
      
      \foreach \j in {0,...,2} {
        \pgfmathsetmacro{\x}{((\j*3))*\wid+0.5*\wid}%
        \pgfmathsetmacro{\y}{(-(\j*3)+19)*\wid+0.5*\wid}%
        \node[] at (\x,\y) {$1$};
      }
      
      \foreach \j in {0,...,2} {
        \pgfmathsetmacro{\x}{((\j*3)+1)*\wid+0.5*\wid}%
        \pgfmathsetmacro{\y}{(-(\j*3)+16)*\wid+0.5*\wid}%
        \node[] at (\x,\y) {$x_9$};
      }
      
      \foreach \j in {0,...,1} {
        \pgfmathsetmacro{\x}{((\j*3)+2)*\wid+0.5*\wid}%
        \pgfmathsetmacro{\y}{(-(\j*3)+16)*\wid+0.5*\wid}%
        \node[] at (\x,\y) {$x_{10}$};
      }
      
      \foreach \j in {0,...,2} {
        \pgfmathsetmacro{\x}{((\j*3)+1)*\wid+0.5*\wid}%
        \pgfmathsetmacro{\y}{(-(\j*3)+13)*\wid+0.5*\wid}%
        \node[] at (\x,\y) {$x_7$};
      }
      
      \foreach \j in {0,...,1} {
        \pgfmathsetmacro{\x}{((\j*3)+2)*\wid+0.5*\wid}%
        \pgfmathsetmacro{\y}{(-(\j*3)+13)*\wid+0.5*\wid}%
        \node[] at (\x,\y) {$x_8$};
      }
      
      \foreach \j in {0,...,2} {
        \pgfmathsetmacro{\x}{((\j*3)+1)*\wid+0.5*\wid}%
        \pgfmathsetmacro{\y}{(-(\j*3)+10)*\wid+0.5*\wid}%
        \node[] at (\x,\y) {$x_5$};
      }
      
      \foreach \j in {0,...,1} {
        \pgfmathsetmacro{\x}{((\j*3)+2)*\wid+0.5*\wid}%
        \pgfmathsetmacro{\y}{(-(\j*3)+10)*\wid+0.5*\wid}%
        \node[] at (\x,\y) {$x_6$};
      }
      
      \foreach \j in {0,...,2} {
        \pgfmathsetmacro{\x}{((\j*3)+1)*\wid+0.5*\wid}%
        \pgfmathsetmacro{\y}{(-(\j*3)+8)*\wid+0.5*\wid}%
        \node[] at (\x,\y) {$x_4$};
        \node[] at (\x+\wid,\y) {$1$};
        \node[] at (\x,\y-\wid) {$x_3$};
      }
      
      \foreach \j in {0,...,1} {
        \pgfmathsetmacro{\x}{((\j*3)+1)*\wid+0.5*\wid}%
        \pgfmathsetmacro{\y}{(-(\j*3)+5)*\wid+0.5*\wid}%
        \node[] at (\x,\y) {$x_2$};
        \node[] at (\x,\y-\wid) {$x_1$};
        \node[] at (\x,\y-2*\wid) {1};
      }
      \draw[blue, thick] (1*\wid, 5*\wid) -- (8*\wid, 5*\wid);
      \draw[blue, thick] (1*\wid, 5*\wid) -- (1*\wid, 20*\wid);
      \draw[blue, thick] (8*\wid, 5*\wid) -- (8*\wid, 20*\wid);
      \begin{scope}[xshift=7cm]
        \draw[LightSkyBlue1, step=\wid] (0,0) grid (4,10);
        
        \foreach  \j  in {1,...,8} {
          \pgfmathtruncatemacro{\x}{\j-6}%
          \node[font=\tiny] at (\j*\wid-0.5*\wid,20.2*\wid) {$\x$};
        };

        \foreach  \j  in {1,...,20} {
          \pgfmathtruncatemacro{\x}{7-\j}%
          \node[font=\tiny] at (-0.3,\j*\wid-0.5*\wid) {$\x$};
        };
        \draw[blue!20, line width=2.5mm, opacity=\al] (0.5*\wid, 5*\wid) -- (0.5*\wid, 20*\wid);
        \draw[blue!20, line width=2.5mm, opacity=\al] (1.5*\wid, 5*\wid) -- (1.5*\wid, 20*\wid);
        \draw[blue!20, line width=2.5mm, opacity=\al] (3.5*\wid, 5*\wid) -- (3.5*\wid, 20*\wid);
        \draw[blue!20, line width=2.5mm, opacity=\al] (4.5*\wid, 5*\wid) -- (4.5*\wid, 20*\wid);
        \draw[blue!20, line width=2.5mm, opacity=\al] (5.5*\wid, 5*\wid) -- (5.5*\wid, 20*\wid);
        \draw[blue!20, line width=2.5mm, opacity=\al] (6.5*\wid, 5*\wid) -- (6.5*\wid, 20*\wid);
        \draw[blue!20, line width=2.5mm, opacity=\al] (7.5*\wid, 5*\wid) -- (7.5*\wid, 20*\wid);
        
        \draw[red!20, line width=2.5mm, opacity=\al] (0*\wid, 5.5*\wid) -- (8*\wid, 5.5*\wid);
        \draw[red!20, line width=2.5mm, opacity=\al] (0*\wid, 7.5*\wid) -- (8*\wid, 7.5*\wid);
        \draw[red!20, line width=2.5mm, opacity=\al] (0*\wid, 8.5*\wid) -- (8*\wid, 8.5*\wid);
        \draw[red!20, line width=2.5mm, opacity=\al] (0*\wid, 10.5*\wid) -- (8*\wid, 10.5*\wid);
        \draw[red!20, line width=2.5mm, opacity=\al] (0*\wid, 13.5*\wid) -- (8*\wid, 13.5*\wid);
        \draw[red!20, line width=2.5mm, opacity=\al] (0*\wid, 16.5*\wid) -- (8*\wid, 16.5*\wid);
        \draw[red!20, line width=2.5mm, opacity=\al] (0*\wid, 19.5*\wid) -- (8*\wid, 19.5*\wid);
        
        \foreach \j in {0,...,2} {
          \pgfmathsetmacro{\x}{((\j*3)+1)*\wid+0.5*\wid}%
          \pgfmathsetmacro{\y}{(-(\j*3)+19)*\wid+0.5*\wid}%
          \node[] at (\x,\y) {$1$};
          \node[] at (\x-\wid, \y){$x_{11}$};
        }
        
        \foreach \j in {0,...,2} {
          \pgfmathsetmacro{\x}{((\j*3))*\wid+0.5*\wid}%
          \pgfmathsetmacro{\y}{(-(\j*3)+16)*\wid+0.5*\wid}%
          \node[] at (\x,\y) {$x_9$};
        }
        
        \foreach \j in {0,...,1} {
          \pgfmathsetmacro{\x}{((\j*3)+2)*\wid+0.5*\wid}%
          \pgfmathsetmacro{\y}{(-(\j*3)+16)*\wid+0.5*\wid}%
          \node[] at (\x,\y) {$x_{10}$};
        }
        
        \foreach \j in {0,...,2} {
          \pgfmathsetmacro{\x}{((\j*3))*\wid+0.5*\wid}%
          \pgfmathsetmacro{\y}{(-(\j*3)+13)*\wid+0.5*\wid}%
          \node[] at (\x,\y) {$x_7$};
        }
        
        \foreach \j in {0,...,1} {
          \pgfmathsetmacro{\x}{((\j*3)+2)*\wid+0.5*\wid}%
          \pgfmathsetmacro{\y}{(-(\j*3)+13)*\wid+0.5*\wid}%
          \node[] at (\x,\y) {$x_8$};
        }
        
        \foreach \j in {0,...,2} {
          \pgfmathsetmacro{\x}{((\j*3))*\wid+0.5*\wid}%
          \pgfmathsetmacro{\y}{(-(\j*3)+10)*\wid+0.5*\wid}%
          \node[] at (\x,\y) {$x_5$};
        }
        
        \foreach \j in {0,...,1} {
          \pgfmathsetmacro{\x}{((\j*3)+2)*\wid+0.5*\wid}%
          \pgfmathsetmacro{\y}{(-(\j*3)+10)*\wid+0.5*\wid}%
          \node[] at (\x,\y) {$x_6$};
        }
        
        \foreach \j in {0,...,2} {
          \pgfmathsetmacro{\x}{((\j*3))*\wid+0.5*\wid}%
          \pgfmathsetmacro{\y}{(-(\j*3)+8)*\wid+0.5*\wid}%
          \node[] at (\x,\y) {$x_4$};
          \node[] at (\x,\y-\wid) {$x_3$};
        }
        
        \foreach \j in {0,...,1} {
          \pgfmathsetmacro{\x}{((\j*3)+1)*\wid+0.5*\wid}%
          \pgfmathsetmacro{\y}{(-(\j*3)+8)*\wid+0.5*\wid}%
          \node[] at (\x+\wid,\y) {$1$};
        }
        
        \foreach \j in {0,...,1} {
          \pgfmathsetmacro{\x}{((\j*3))*\wid+0.5*\wid}%
          \pgfmathsetmacro{\y}{(-(\j*3)+5)*\wid+0.5*\wid}%
          \node[] at (\x,\y) {$x_2$};
          \node[] at (\x,\y-\wid) {$x_1$};
          \node[] at (\x,\y-2*\wid) {1};
        }
        \draw[blue, thick] (0*\wid, 5*\wid) -- (8*\wid, 5*\wid);
        \draw[blue, thick] (0*\wid, 5*\wid) -- (0*\wid, 20*\wid);
        \draw[blue, thick] (8*\wid, 5*\wid) -- (8*\wid, 20*\wid);
      \end{scope}
    \end{tikzpicture}

    \caption{Example computation for Case (a)}
    \label{fig:Ex(a)}
  \end{figure}
\end{example}

\begin{example}
  \label{ex:Case(b)}
  (Example computation for Case (\ref{thm:caseb})) Let $w=[9,-7,4]$ and $v=[5,0,1]$. We compute \[Q(w)=s_2s_1s_0s_2s_1s_2s_0s_1s_2s_0s_1.\]
  We take $(i,j)=(1,1)\in\ess(v)$. We compute $l_{ws_\sigma}(1,1)=6$ and $l_w(1,1)=7$ (see Figure \ref{fig:Ex(b)}). We also compute $n_v(1,1,6)=4$ and $n_v(1,1,7)=5$. Consider $g=\det M_{A,B}$ where $A=\{ -10,-7,-4,-1,1 \}$ and $B=\{ -4,-3,-2,0,1 \}.$
  As in Example \ref{ex:Case(a)}, we may interpret $g$ as a $6\times 6$ minor in $M_{(-\infty, 1],[-5,1]}$ by including row $-13$ and column $-5$.
  \[
    A=\{ -13,-10,-7,-4,-1,1 \}\text{ and }B=\{ -5,-4,-3,-2,0,1 \}.
  \]
  We follow the algorithm; note that $\sigma=1$.
  \begin{itemize}
  \item[Step 1.] We replace $-2$ by $-1$, and perform $\bsswap(-2)$. Note that when we get to column $1$ (which is in $B$), we do nothing.
  \item[Step 2.] There is nothing to do in this step.
  \item[Step 3.] We perform $\bsswap(-5)$ and $\bsswap(1)$ and we are done, since the remaining $\bsswap$-s do not affect the submatrix. 
  \end{itemize}
  After the algorithm, we have
  \[
    A^\prime=\{ -13,-10,-7,-4,-1,1 \}\text{ and }B^\prime=\{ -5,-4,-3,-1,0,1 \}.
  \]
  We have $\LT(\det(M^\prime)_{A^\prime,B^\prime})=x_{11}x_6^2x_2$, which indeed divides $(x_{11})\LT(\det(M)_{A,B})=-x_{11}x_6^2x_2$.
    \begin{figure}[h!]
    \centering
    \begin{tikzpicture}[scale=0.8,transform shape]
      \pgfmathsetmacro{\wid}{0.5}
      \pgfmathsetmacro{\al}{0.4}
      \draw[LightSkyBlue1, step=\wid] (0,0) grid (4,10);
      
      \foreach  \j  in {1,...,8} {
        \pgfmathtruncatemacro{\x}{\j-6}%
        \node[font=\tiny] at (\j*\wid-0.5*\wid,20.2*\wid) {$\x$};
      };

      \foreach  \j  in {1,...,20} {
        \pgfmathtruncatemacro{\x}{7-\j}%
        \node[font=\tiny] at (-0.3,\j*\wid-0.5*\wid) {$\x$};
      };

      \draw[blue!20, line width=2.5mm, opacity=\al] (0.5*\wid, 5*\wid) -- (0.5*\wid, 20*\wid);
      \draw[blue!20, line width=2.5mm, opacity=\al] (1.5*\wid, 5*\wid) -- (1.5*\wid, 20*\wid);
      \draw[blue!20, line width=2.5mm, opacity=\al] (2.5*\wid, 5*\wid) -- (2.5*\wid, 20*\wid);
      \draw[blue!20, line width=2.5mm, opacity=\al] (3.5*\wid, 5*\wid) -- (3.5*\wid, 20*\wid);
      \draw[blue!20, line width=2.5mm, opacity=\al] (5.5*\wid, 5*\wid) -- (5.5*\wid, 20*\wid);
      \draw[blue!20, line width=2.5mm, opacity=\al] (6.5*\wid, 5*\wid) -- (6.5*\wid, 20*\wid);
      
      \draw[red!20, line width=2.5mm, opacity=\al] (0*\wid, 5.5*\wid) -- (7*\wid, 5.5*\wid);
      \draw[red!20, line width=2.5mm, opacity=\al] (0*\wid, 7.5*\wid) -- (7*\wid, 7.5*\wid);
      \draw[red!20, line width=2.5mm, opacity=\al] (0*\wid, 10.5*\wid) -- (7*\wid, 10.5*\wid);
      \draw[red!20, line width=2.5mm, opacity=\al] (0*\wid, 13.5*\wid) -- (7*\wid, 13.5*\wid);
      \draw[red!20, line width=2.5mm, opacity=\al] (0*\wid, 16.5*\wid) -- (7*\wid, 16.5*\wid);
      \draw[red!20, line width=2.5mm, opacity=\al] (0*\wid, 19.5*\wid) -- (7*\wid, 19.5*\wid);
      
      \foreach \j in {0,...,2} {
        \pgfmathsetmacro{\x}{((\j*3))*\wid+0.5*\wid}%
        \pgfmathsetmacro{\y}{(-(\j*3)+19)*\wid+0.5*\wid}%
        \node[] at (\x,\y) {$1$};
      }
      
      \foreach \j in {0,...,2} {
        \pgfmathsetmacro{\x}{((\j*3)+1)*\wid+0.5*\wid}%
        \pgfmathsetmacro{\y}{(-(\j*3)+16)*\wid+0.5*\wid}%
        \node[] at (\x,\y) {$x_9$};
      }
      
      \foreach \j in {0,...,1} {
        \pgfmathsetmacro{\x}{((\j*3)+2)*\wid+0.5*\wid}%
        \pgfmathsetmacro{\y}{(-(\j*3)+16)*\wid+0.5*\wid}%
        \node[] at (\x,\y) {$x_{10}$};
      }
      
      \foreach \j in {0,...,2} {
        \pgfmathsetmacro{\x}{((\j*3)+1)*\wid+0.5*\wid}%
        \pgfmathsetmacro{\y}{(-(\j*3)+13)*\wid+0.5*\wid}%
        \node[] at (\x,\y) {$x_7$};
      }
      
      \foreach \j in {0,...,1} {
        \pgfmathsetmacro{\x}{((\j*3)+2)*\wid+0.5*\wid}%
        \pgfmathsetmacro{\y}{(-(\j*3)+13)*\wid+0.5*\wid}%
        \node[] at (\x,\y) {$x_8$};
      }
      
      \foreach \j in {0,...,2} {
        \pgfmathsetmacro{\x}{((\j*3)+1)*\wid+0.5*\wid}%
        \pgfmathsetmacro{\y}{(-(\j*3)+10)*\wid+0.5*\wid}%
        \node[] at (\x,\y) {$x_5$};
      }
      
      \foreach \j in {0,...,1} {
        \pgfmathsetmacro{\x}{((\j*3)+2)*\wid+0.5*\wid}%
        \pgfmathsetmacro{\y}{(-(\j*3)+10)*\wid+0.5*\wid}%
        \node[] at (\x,\y) {$x_6$};
      }
      
      \foreach \j in {0,...,2} {
        \pgfmathsetmacro{\x}{((\j*3)+1)*\wid+0.5*\wid}%
        \pgfmathsetmacro{\y}{(-(\j*3)+8)*\wid+0.5*\wid}%
        \node[] at (\x,\y) {$x_4$};
        \node[] at (\x+\wid,\y) {$1$};
        \node[] at (\x,\y-\wid) {$x_3$};
      }
      
      \foreach \j in {0,...,1} {
        \pgfmathsetmacro{\x}{((\j*3)+1)*\wid+0.5*\wid}%
        \pgfmathsetmacro{\y}{(-(\j*3)+5)*\wid+0.5*\wid}%
        \node[] at (\x,\y) {$x_2$};
        \node[] at (\x,\y-\wid) {$x_1$};
        \node[] at (\x,\y-2*\wid) {1};
      }
      \draw[blue, thick] (1*\wid, 5*\wid) -- (7*\wid, 5*\wid);
      \draw[blue, thick] (1*\wid, 5*\wid) -- (1*\wid, 20*\wid);
      \draw[blue, thick] (7*\wid, 5*\wid) -- (7*\wid, 20*\wid);
      \begin{scope}[xshift=7cm]
        \draw[LightSkyBlue1, step=\wid] (0,0) grid (4,10);
        
        \foreach  \j  in {1,...,8} {
          \pgfmathtruncatemacro{\x}{\j-6}%
          \node[font=\tiny] at (\j*\wid-0.5*\wid,20.2*\wid) {$\x$};
        };

        \foreach  \j  in {1,...,20} {
          \pgfmathtruncatemacro{\x}{7-\j}%
          \node[font=\tiny] at (-0.3,\j*\wid-0.5*\wid) {$\x$};
        };
        \draw[blue!20, line width=2.5mm, opacity=\al] (0.5*\wid, 5*\wid) -- (0.5*\wid, 20*\wid);
        \draw[blue!20, line width=2.5mm, opacity=\al] (1.5*\wid, 5*\wid) -- (1.5*\wid, 20*\wid);
        \draw[blue!20, line width=2.5mm, opacity=\al] (2.5*\wid, 5*\wid) -- (2.5*\wid, 20*\wid);
        \draw[blue!20, line width=2.5mm, opacity=\al] (4.5*\wid, 5*\wid) -- (4.5*\wid, 20*\wid);
        \draw[blue!20, line width=2.5mm, opacity=\al] (5.5*\wid, 5*\wid) -- (5.5*\wid, 20*\wid);
        \draw[blue!20, line width=2.5mm, opacity=\al] (6.5*\wid, 5*\wid) -- (6.5*\wid, 20*\wid);
        
        \draw[red!20, line width=2.5mm, opacity=\al] (0*\wid, 5.5*\wid) -- (7*\wid, 5.5*\wid);
        \draw[red!20, line width=2.5mm, opacity=\al] (0*\wid, 7.5*\wid) -- (7*\wid, 7.5*\wid);
        \draw[red!20, line width=2.5mm, opacity=\al] (0*\wid, 10.5*\wid) -- (7*\wid, 10.5*\wid);
        \draw[red!20, line width=2.5mm, opacity=\al] (0*\wid, 13.5*\wid) -- (7*\wid, 13.5*\wid);
        \draw[red!20, line width=2.5mm, opacity=\al] (0*\wid, 16.5*\wid) -- (7*\wid, 16.5*\wid);
        \draw[red!20, line width=2.5mm, opacity=\al] (0*\wid, 19.5*\wid) -- (7*\wid, 19.5*\wid);
        
        \foreach \j in {0,...,2} {
          \pgfmathsetmacro{\x}{((\j*3)+1)*\wid+0.5*\wid}%
          \pgfmathsetmacro{\y}{(-(\j*3)+19)*\wid+0.5*\wid}%
          \node[] at (\x,\y) {$1$};
          \node[] at (\x-\wid, \y){$x_{11}$};
        }
        
        \foreach \j in {0,...,2} {
          \pgfmathsetmacro{\x}{((\j*3))*\wid+0.5*\wid}%
          \pgfmathsetmacro{\y}{(-(\j*3)+16)*\wid+0.5*\wid}%
          \node[] at (\x,\y) {$x_9$};
        }
        
        \foreach \j in {0,...,1} {
          \pgfmathsetmacro{\x}{((\j*3)+2)*\wid+0.5*\wid}%
          \pgfmathsetmacro{\y}{(-(\j*3)+16)*\wid+0.5*\wid}%
          \node[] at (\x,\y) {$x_{10}$};
        }
        
        \foreach \j in {0,...,2} {
          \pgfmathsetmacro{\x}{((\j*3))*\wid+0.5*\wid}%
          \pgfmathsetmacro{\y}{(-(\j*3)+13)*\wid+0.5*\wid}%
          \node[] at (\x,\y) {$x_7$};
        }
        
        \foreach \j in {0,...,1} {
          \pgfmathsetmacro{\x}{((\j*3)+2)*\wid+0.5*\wid}%
          \pgfmathsetmacro{\y}{(-(\j*3)+13)*\wid+0.5*\wid}%
          \node[] at (\x,\y) {$x_8$};
        }
        
        \foreach \j in {0,...,2} {
          \pgfmathsetmacro{\x}{((\j*3))*\wid+0.5*\wid}%
          \pgfmathsetmacro{\y}{(-(\j*3)+10)*\wid+0.5*\wid}%
          \node[] at (\x,\y) {$x_5$};
        }
        
        \foreach \j in {0,...,1} {
          \pgfmathsetmacro{\x}{((\j*3)+2)*\wid+0.5*\wid}%
          \pgfmathsetmacro{\y}{(-(\j*3)+10)*\wid+0.5*\wid}%
          \node[] at (\x,\y) {$x_6$};
        }
        
        \foreach \j in {0,...,2} {
          \pgfmathsetmacro{\x}{((\j*3))*\wid+0.5*\wid}%
          \pgfmathsetmacro{\y}{(-(\j*3)+8)*\wid+0.5*\wid}%
          \node[] at (\x,\y) {$x_4$};
          \node[] at (\x,\y-\wid) {$x_3$};
        }
        
        \foreach \j in {0,...,1} {
          \pgfmathsetmacro{\x}{((\j*3)+1)*\wid+0.5*\wid}%
          \pgfmathsetmacro{\y}{(-(\j*3)+8)*\wid+0.5*\wid}%
          \node[] at (\x+\wid,\y) {$1$};
        }
        
        \foreach \j in {0,...,1} {
          \pgfmathsetmacro{\x}{((\j*3))*\wid+0.5*\wid}%
          \pgfmathsetmacro{\y}{(-(\j*3)+5)*\wid+0.5*\wid}%
          \node[] at (\x,\y) {$x_2$};
          \node[] at (\x,\y-\wid) {$x_1$};
          \node[] at (\x,\y-2*\wid) {1};
        }
        \draw[blue, thick] (0*\wid, 5*\wid) -- (7*\wid, 5*\wid);
        \draw[blue, thick] (0*\wid, 5*\wid) -- (0*\wid, 20*\wid);
        \draw[blue, thick] (7*\wid, 5*\wid) -- (7*\wid, 20*\wid);
      \end{scope}
    \end{tikzpicture}

    \caption{Example computation for Case (b)}
    \label{fig:Ex(b)}
  \end{figure}
\end{example}

\begin{example}
  \label{ex:Case(c)2}
  (Example computation for Case (\ref{thm:casec})) Let $w=[9,-7,4]$ and $v=[5,0,1]$. We compute \[Q(w)=s_2s_1s_0s_2s_1s_2s_0s_1s_2s_0s_1.\] Then $vs_1=[0,5,1]$. We take $(i,j):=(1,2)\in\ess(vs_1)$, and $(i,j-1)=(1,1)\in \ess(v)$. Then $l_{ws_\sigma}(1,2)=l_w(1,1)=7$, and $n_{vs_\sigma}(1,2,7)=n_v(1,1,7)=5$, so $EQ_{ws_1,vs_1}^{1,2}$ consists of 6 by 6 minors of $M_{(-\infty, 1],[-4,2]}$, and $EQ_{w,v}^{1,1}$ consists of 6 by 6 minors of $M'_{(-\infty, 1],[-5,1]}$ as shown in Figure~\ref{fig:Ex(c)2}.
  Consider $g:=\det M_{A,B}$, where
  \[
    A=\{-10,-7,-4,-2,-1,1\} \text{ and }  B=\{-4,-3,-2,-1,0,1\}.
  \]
  We follow the algorithm with $\sigma=1$. We are in the case where $(i,j-1)\in\ess(v)$, and in the subcase where $j-l_{ws_\sigma}(i,j)=-5\in [\sigma]_n$.
  \begin{itemize}
  \item[Step 1.] There is nothing to do until we get to the special step, since $j-1=1\in B$ but $2\not\in B$.  We replace $1$ by $1-7+1=-5$ and we replace $ws_\sigma(1)=-7$ in $A$ by $ws_\sigma(-5)=-13$, and we perform $\bsswap(1)$. 
  \item[Step 3.] We perform $\bsswap(-5)$ and $\bsswap(-2)$ and we are done as the other $\bsswap$-s do not affect the submatrix.
  \end{itemize}
  After the algorithm we have
  \[
    A'=\{-13,-10,-4,-2,-1,1\} \text{ and }  B'=\{-5,-4,-3,-2,-1,0\}.
  \]
  Note that $\LT(\det(M^\prime)_{A',B'})=-x_8x_5x_2=\LT(\det(M)_{A,B})$.

  \begin{figure}[h!]
    \centering
    \begin{tikzpicture}[scale=0.8,transform shape]
      \pgfmathsetmacro{\wid}{0.5}
      \pgfmathsetmacro{\al}{0.4}
      \draw[LightSkyBlue1, step=\wid] (0,0) grid (4,10);
      
      \foreach  \j  in {1,...,8} {
        \pgfmathtruncatemacro{\x}{\j-6}%
        \node[font=\tiny] at (\j*\wid-0.5*\wid,20.2*\wid) {$\x$};
      };

      \foreach  \j  in {1,...,20} {
        \pgfmathtruncatemacro{\x}{7-\j}%
        \node[font=\tiny] at (-0.3,\j*\wid-0.5*\wid) {$\x$};
      };
      
      \draw[blue!20, line width=2.5mm, opacity=\al] (1.5*\wid, 5*\wid) -- (1.5*\wid, 20*\wid);
      \draw[blue!20, line width=2.5mm, opacity=\al] (2.5*\wid, 5*\wid) -- (2.5*\wid, 20*\wid);
      \draw[blue!20, line width=2.5mm, opacity=\al] (3.5*\wid, 5*\wid) -- (3.5*\wid, 20*\wid);
      \draw[blue!20, line width=2.5mm, opacity=\al] (4.5*\wid, 5*\wid) -- (4.5*\wid, 20*\wid);
      \draw[blue!20, line width=2.5mm, opacity=\al] (5.5*\wid, 5*\wid) -- (5.5*\wid, 20*\wid);
      \draw[blue!20, line width=2.5mm, opacity=\al] (6.5*\wid, 5*\wid) -- (6.5*\wid, 20*\wid);
      
      \draw[red!20, line width=2.5mm, opacity=\al] (1*\wid, 5.5*\wid) -- (8*\wid, 5.5*\wid);
      \draw[red!20, line width=2.5mm, opacity=\al] (1*\wid, 7.5*\wid) -- (8*\wid, 7.5*\wid);
      \draw[red!20, line width=2.5mm, opacity=\al] (1*\wid, 8.5*\wid) -- (8*\wid, 8.5*\wid);
      \draw[red!20, line width=2.5mm, opacity=\al] (1*\wid, 10.5*\wid) -- (8*\wid, 10.5*\wid);
      \draw[red!20, line width=2.5mm, opacity=\al] (1*\wid, 13.5*\wid) -- (8*\wid, 13.5*\wid);
      \draw[red!20, line width=2.5mm, opacity=\al] (1*\wid, 16.5*\wid) -- (8*\wid, 16.5*\wid);
      \foreach \j in {0,...,2} {
        \pgfmathsetmacro{\x}{((\j*3))*\wid+0.5*\wid}%
        \pgfmathsetmacro{\y}{(-(\j*3)+19)*\wid+0.5*\wid}%
        \node[] at (\x,\y) {$1$};
      }
      
      \foreach \j in {0,...,2} {
        \pgfmathsetmacro{\x}{((\j*3)+1)*\wid+0.5*\wid}%
        \pgfmathsetmacro{\y}{(-(\j*3)+16)*\wid+0.5*\wid}%
        \node[] at (\x,\y) {$x_9$};
      }
      
      \foreach \j in {0,...,1} {
        \pgfmathsetmacro{\x}{((\j*3)+2)*\wid+0.5*\wid}%
        \pgfmathsetmacro{\y}{(-(\j*3)+16)*\wid+0.5*\wid}%
        \node[] at (\x,\y) {$x_{10}$};
      }
      
      \foreach \j in {0,...,2} {
        \pgfmathsetmacro{\x}{((\j*3)+1)*\wid+0.5*\wid}%
        \pgfmathsetmacro{\y}{(-(\j*3)+13)*\wid+0.5*\wid}%
        \node[] at (\x,\y) {$x_7$};
      }
      
      \foreach \j in {0,...,1} {
        \pgfmathsetmacro{\x}{((\j*3)+2)*\wid+0.5*\wid}%
        \pgfmathsetmacro{\y}{(-(\j*3)+13)*\wid+0.5*\wid}%
        \node[] at (\x,\y) {$x_8$};
      }
      
      \foreach \j in {0,...,2} {
        \pgfmathsetmacro{\x}{((\j*3)+1)*\wid+0.5*\wid}%
        \pgfmathsetmacro{\y}{(-(\j*3)+10)*\wid+0.5*\wid}%
        \node[] at (\x,\y) {$x_5$};
      }
      
      \foreach \j in {0,...,1} {
        \pgfmathsetmacro{\x}{((\j*3)+2)*\wid+0.5*\wid}%
        \pgfmathsetmacro{\y}{(-(\j*3)+10)*\wid+0.5*\wid}%
        \node[] at (\x,\y) {$x_6$};
      }
      
      \foreach \j in {0,...,2} {
        \pgfmathsetmacro{\x}{((\j*3)+1)*\wid+0.5*\wid}%
        \pgfmathsetmacro{\y}{(-(\j*3)+8)*\wid+0.5*\wid}%
        \node[] at (\x,\y) {$x_4$};
        \node[] at (\x+\wid,\y) {$1$};
        \node[] at (\x,\y-\wid) {$x_3$};
      }
      
      \foreach \j in {0,...,1} {
        \pgfmathsetmacro{\x}{((\j*3)+1)*\wid+0.5*\wid}%
        \pgfmathsetmacro{\y}{(-(\j*3)+5)*\wid+0.5*\wid}%
        \node[] at (\x,\y) {$x_2$};
        \node[] at (\x,\y-\wid) {$x_1$};
        \node[] at (\x,\y-2*\wid) {1};
      }
      \draw[blue, thick] (1*\wid, 5*\wid) -- (8*\wid, 5*\wid);
      \draw[blue, thick] (1*\wid, 5*\wid) -- (1*\wid, 20*\wid);
      \draw[blue, thick] (8*\wid, 5*\wid) -- (8*\wid, 20*\wid);
      
      \begin{scope}[xshift=7cm]
        \draw[LightSkyBlue1, step=\wid] (0,0) grid (4,10);
        
        \foreach  \j  in {1,...,8} {
          \pgfmathtruncatemacro{\x}{\j-6}%
          \node[font=\tiny] at (\j*\wid-0.5*\wid,20.2*\wid) {$\x$};
        };

        \foreach  \j  in {1,...,20} {
          \pgfmathtruncatemacro{\x}{7-\j}%
          \node[font=\tiny] at (-0.3,\j*\wid-0.5*\wid) {$\x$};
        };
        \draw[blue!20, line width=2.5mm, opacity=\al] (0.5*\wid, 5*\wid) -- (0.5*\wid, 20*\wid);
        \draw[blue!20, line width=2.5mm, opacity=\al] (1.5*\wid, 5*\wid) -- (1.5*\wid, 20*\wid);
        \draw[blue!20, line width=2.5mm, opacity=\al] (2.5*\wid, 5*\wid) -- (2.5*\wid, 20*\wid);
        \draw[blue!20, line width=2.5mm, opacity=\al] (3.5*\wid, 5*\wid) -- (3.5*\wid, 20*\wid);
        \draw[blue!20, line width=2.5mm, opacity=\al] (4.5*\wid, 5*\wid) -- (4.5*\wid, 20*\wid);
        \draw[blue!20, line width=2.5mm, opacity=\al] (5.5*\wid, 5*\wid) -- (5.5*\wid, 20*\wid);

        \draw[red!20, line width=2.5mm, opacity=\al] (0*\wid, 5.5*\wid) -- (7*\wid, 5.5*\wid);
        \draw[red!20, line width=2.5mm, opacity=\al] (0*\wid, 7.5*\wid) -- (7*\wid, 7.5*\wid);
        \draw[red!20, line width=2.5mm, opacity=\al] (0*\wid, 8.5*\wid) -- (7*\wid, 8.5*\wid);
        \draw[red!20, line width=2.5mm, opacity=\al] (0*\wid, 10.5*\wid) -- (7*\wid, 10.5*\wid);
        \draw[red!20, line width=2.5mm, opacity=\al] (0*\wid, 19.5*\wid) -- (7*\wid, 19.5*\wid);
        \draw[red!20, line width=2.5mm, opacity=\al] (0*\wid, 16.5*\wid) -- (7*\wid, 16.5*\wid);
        
        \foreach \j in {0,...,2} {
          \pgfmathsetmacro{\x}{((\j*3)+1)*\wid+0.5*\wid}%
          \pgfmathsetmacro{\y}{(-(\j*3)+19)*\wid+0.5*\wid}%
          \node[] at (\x,\y) {$1$};
          \node[] at (\x-\wid, \y){$x_{11}$};
        }
        
        \foreach \j in {0,...,2} {
          \pgfmathsetmacro{\x}{((\j*3))*\wid+0.5*\wid}%
          \pgfmathsetmacro{\y}{(-(\j*3)+16)*\wid+0.5*\wid}%
          \node[] at (\x,\y) {$x_9$};
        }
        
        \foreach \j in {0,...,1} {
          \pgfmathsetmacro{\x}{((\j*3)+2)*\wid+0.5*\wid}%
          \pgfmathsetmacro{\y}{(-(\j*3)+16)*\wid+0.5*\wid}%
          \node[] at (\x,\y) {$x_{10}$};
        }
        
        \foreach \j in {0,...,2} {
          \pgfmathsetmacro{\x}{((\j*3))*\wid+0.5*\wid}%
          \pgfmathsetmacro{\y}{(-(\j*3)+13)*\wid+0.5*\wid}%
          \node[] at (\x,\y) {$x_7$};
        }
        
        \foreach \j in {0,...,1} {
          \pgfmathsetmacro{\x}{((\j*3)+2)*\wid+0.5*\wid}%
          \pgfmathsetmacro{\y}{(-(\j*3)+13)*\wid+0.5*\wid}%
          \node[] at (\x,\y) {$x_8$};
        }
        
        \foreach \j in {0,...,2} {
          \pgfmathsetmacro{\x}{((\j*3))*\wid+0.5*\wid}%
          \pgfmathsetmacro{\y}{(-(\j*3)+10)*\wid+0.5*\wid}%
          \node[] at (\x,\y) {$x_5$};
        }
        
        \foreach \j in {0,...,1} {
          \pgfmathsetmacro{\x}{((\j*3)+2)*\wid+0.5*\wid}%
          \pgfmathsetmacro{\y}{(-(\j*3)+10)*\wid+0.5*\wid}%
          \node[] at (\x,\y) {$x_6$};
        }
        
        \foreach \j in {0,...,2} {
          \pgfmathsetmacro{\x}{((\j*3))*\wid+0.5*\wid}%
          \pgfmathsetmacro{\y}{(-(\j*3)+8)*\wid+0.5*\wid}%
          \node[] at (\x,\y) {$x_4$};
          \node[] at (\x,\y-\wid) {$x_3$};
        }
        
        \foreach \j in {0,...,1} {
          \pgfmathsetmacro{\x}{((\j*3)+1)*\wid+0.5*\wid}%
          \pgfmathsetmacro{\y}{(-(\j*3)+8)*\wid+0.5*\wid}%
          \node[] at (\x+\wid,\y) {$1$};
        }
        
        \foreach \j in {0,...,1} {
          \pgfmathsetmacro{\x}{((\j*3))*\wid+0.5*\wid}%
          \pgfmathsetmacro{\y}{(-(\j*3)+5)*\wid+0.5*\wid}%
          \node[] at (\x,\y) {$x_2$};
          \node[] at (\x,\y-\wid) {$x_1$};
          \node[] at (\x,\y-2*\wid) {1};
        }
        \draw[blue, thick] (0*\wid, 5*\wid) -- (7*\wid, 5*\wid);
        \draw[blue, thick] (0*\wid, 5*\wid) -- (0*\wid, 20*\wid);
        \draw[blue, thick] (7*\wid, 5*\wid) -- (7*\wid, 20*\wid);
      \end{scope}
    \end{tikzpicture}

    \caption{Example computation for Case (c)}
    \label{fig:Ex(c)2}
  \end{figure}
\end{example}

\begin{proof}[Proof of Proposition~\ref{prop:SRinJ}]
 We proceed by induction on $\ell(w)$. The base case for $\ell(w)=0$ is trivial. Now assume $\ell(w)>0$. Suppose $s_\sigma$ is the last simple reflection of $Q(w)$, namely, $j=\sigma(w)$. 
 
 Suppose $P\subset \{1,2,\cdots,\ell(w)\}$ such that $P\not\in\Delta(Q(w),v)$. We  show that $\prod_{a\in P}x_a$ is divisible by  $\LT(g)$ for some $g\in EQ_{w,v}$. Let $q=\ell(w)$. 
 
 \textbf{Case $q\in P$.} By Lemma~\ref{lem:simplicial_induction}(a), $P\setminus\{q\}\not\in \link_q(\Delta(Q(w),v))$, and by Lemma~\ref{lem:subword_induction}, $\link_q(\Delta(Q(w),v))=\Delta(Q(ws_\sigma),v)$. By induction hypothesis,
 $\prod_{a\in P\setminus\{q\}} x_a$ is divisible by $\LT(g)$ for some $g\in EQ_{ws_\sigma,v}$.
 
 \textbf{Subcase $\ell(vs_\sigma>v)$.} By Lemma~\ref{lem:main}(a), there exists $g'\in EQ_{w,v}$ such that $\LT(g')$ divides $\LT(g)$, and hence $\LT(g')$ divides $\prod_{a\in P}x_a$.
 
 \textbf{Subcase $\ell(vs_\sigma>v)$.} By Lemma~\ref{lem:main}(b), there exists $g'\in EQ_{w,v}$ such that $\LT(g')$ divides $x_{q}\LT(g)$. Therefore $\LT(g')$ divides $x_q \prod_{a\in P\setminus\{q\}}x_a=\prod_{a\in P}x_a$.
 
 \textbf{Case $q\not\in P$.} By Lemma~\ref{lem:subword_induction}(b), $P\not\in \del_q(\Delta(Q(w),v))$. 
 
 \textbf{Subcase $\ell(vs_\sigma)>v$.} By Lemma~\ref{lem:subword_induction}, $\del_q(\Delta(Q(w),v))=\Delta(Q(ws_\sigma), v)$. By induction hypothesis, $\prod_{a\in P} x_a$ is divisible by $\LT(g)$ for some $g\in EQ_{ws_\sigma,v}$. The claim follows by Lemma~\ref{lem:main}(a).
 
 \textbf{Subcase $\ell(vs_\sigma)<v$.} By Lemma~\ref{lem:subword_induction}, $\del_q(\Delta(Q(w),v))=\Delta(Q(ws_\sigma), vs_\sigma)$. By induction hypothesis, $\prod_{a\in P} x_a$ is divisible by $\LT(g)$ for some $g\in EQ_{ws_\sigma,vs_\sigma}$. The claim follows by Lemma~\ref{lem:main}(c).
 
\end{proof}

\appendix
\section{}
\label{app}
As noted in Remark~\ref{rmk:KLinit},
Theorem~\ref{thm:knutson} is only stated  in the literature for finite type
flag varieties, and hence the appendix sketches how the proof
can be modified for the case at hand. 

Let  $(G,B,B_-,T,W)$ be a pinning for a symmetrizable Kac-Moody group $G$ over $\C$, completed along the positive roots. 
For every $w\in W$, define $X^w_\circ:=BwB/B$, $X^w:=\overline{X^w_\circ}$,  $X_w^\circ:=B_-wB/B$, $X_w:=\overline{X_w^\circ}$ and $X^v_{w,\circ}:=X^{v}_\circ\cap X_w$. We refer to \cite{kumar2017positivity} for basic properties of Richardson varieties. In this appendix we sketch an argument for \cite[Theorem 2']{knutson2008schubert}, which we restate below. We follow Knutson's notation $v\ge w$, which is the opposite of the convention in our paper. For $\alpha$ a simple root, if $vr_\alpha<v$, 
the Schubert cell $X^v_\circ$ factors $T$-equivariantly as a product of a line $L$ with weight $-v\cdot \alpha$ and a complementary hyperplane $H\cong X^{vr_\alpha}_\circ$. Let $\C^\times$ act on $X^v_\circ$ by scaling $L$. 

\begin{thm}\cite[Theorem 2']{knutson2008schubert}
    Let $v,w\in W$, $v\ge w$. If $v=1$ then $w=1$ and $X^v_{w,\circ} = B/B$. Otherwise, there is a simple root $\alpha$ such that $vr_{\alpha}<v$ in Bruhat order. 
    \begin{enumerate}[(1)]
        \item If $wr_{\alpha}>w$, then $X^v_{w,\circ}\cong X^{vr_{\alpha}}_{w,\circ} \times \mathbb{A}^1_{-v\cdot\alpha}$.
        \item If $wr_\alpha < w$ but $w\not\le vr_\alpha$, then $X^v_{w,\circ}\cong X^{vr_\alpha}_{wr_\alpha,\circ}$.
        \item If $wr_\alpha < w \le vr_\alpha$, then the scheme-theoretic limit $X':=\lim_{t\to 0} X^v_{w,\circ}$ is reduced and has two components:
        \[X'=(\Pi \times \{0\})\cup_{\Lambda\times \{0\}}(\Lambda\times \mathbb{A}^1_{-v\cdot\alpha}).\]
        where $\Pi\cong X^{vr_\alpha}_{wr_\alpha,\circ}$ and $\Lambda\cong X^{vr_\alpha}_{w,\circ}$.
    \end{enumerate}
\end{thm}
\begin{proof}[Proof sketch]
    Let $P_\alpha$ be the standard parabolic subgroup associated to $\alpha$, and let   $\pi_\alpha: G/B\to G/P_\alpha$ denote the natural projection. 
    We have $\pi_\alpha^{-1}(\pi_\alpha(X^v_\circ))=X^v_\circ \cup X^{vr_\alpha}_\circ$. In other words,
    the union $X^v_\circ \cup X^{vr_\alpha}_\circ$ is a trivial $\mathbb{P}^1$-bundle over $\pi_\alpha(X^v_\circ)=\pi_\alpha(X^{vr_\alpha}_\circ) \cong X^{vr_\alpha}_\circ$. 

    For case (1), since $w\le v$, $wr_\alpha >w$, we have $w<vr_\alpha$. Observe that since $\pi_\alpha^{-1}(\pi_\alpha(X^{vr_\alpha}_{w,\circ}))=X^v_{w,\circ}\cup X^{vr_\alpha}_{w,\circ}$ and $\pi_\alpha(X^{vr_\alpha}_{w,\circ})\cong X^{vr_\alpha}_{w,\circ}$, we have $X^v_{w,\circ}\cup X^{vr_\alpha}_{w,\circ}\cong X^{vr_\alpha}_{w,\circ} \times \mathbb{P}^1$. It follows that $X^v_{w,\circ}\cong X^{vr_{\alpha}}_{w,\circ} \times \mathbb{A}^1_{-v\cdot\alpha}$.

    For case (2), the conditions on $v$ and $w$ ensure that $\pi_\alpha$ is an isomorphism restricted to $X^v_{w,\circ}$, and the image is exactly $X^{vr_\alpha}_{wr_\alpha,\circ}$.

    For case (3), we apply the geometric vertex decomposition lemma \cite[Theorem 2]{knutson2007automatically} on $X:=X_w\cap(X^v_\circ\cup X^{vr_\alpha}_\circ)=X^{vr_\alpha}_{w,\circ}\cup X^v_{w,\circ}$  inside the ambient space $X^v_\circ\cup X^{vr_\alpha}_\circ\cong X^{vr_\alpha}_\circ\times \mathbb{P}^1 $. Since $w>wr_\alpha$, $\pi_\alpha(X)$ is generically one-to-one onto its image, which is isomorphic to $X^{vr_\alpha}_{wr_\alpha,\circ}$. It follows that $\Pi\cong X^{vr_\alpha}_{wr_\alpha,\circ}$. 
    Furthermore, under the identification $X^v_\circ\cup X^{vr_\alpha}_\circ\cong X^{vr_\alpha}_\circ\times \mathbb{P}^1 $, the Bruhat cell $X^{vr_\alpha}_\circ$ is mapped to $X^{vr_\alpha}_\circ\times \{\infty\}$, and therefore $\Lambda=X\cap (X^{vr_\alpha}_\circ\times \{\infty\})\cong X^{vr_\alpha}_{w,\circ} $. The normality and reducedness conditions required for applying the geometric vertex decomposition lemma follow from standard facts about Richardson varieties \cite{kumar2017positivity}.
\end{proof}
\bibliographystyle{alpha}
  \bibliography{ref}
    \end{document}